\documentclass[]{birkjour}

\usepackage{epstopdf}% To incorporate .eps illustrations using PDFLaTeX, etc.
\usepackage[caption=false]{subfig}% Support for small, `sub' figures and tables
\usepackage[numbers,sort&compress]{natbib}% Citation support using natbib.sty
\bibpunct[, ]{[}{]}{,}{n}{,}{,}% Citation support using natbib.sty
\makeatletter% @ becomes a letter
\def\NAT@def@citea{\def\@citea{\NAT@separator}}% Suppress spaces between citations using natbib.sty
\makeatother% @ becomes a symbol again
\usepackage{hyperref}
\hypersetup{
	colorlinks=true,
	linkcolor=blue, % Couleur des liens internes
	citecolor=red, % Couleur des num?ros de la biblio dans le corps
	urlcolor=blue  } % Couleur des url
\usepackage{mathptmx}   %Times New Roman

\theoremstyle{plain}% Theorem-like structures provided by amsthm.sty
\newtheorem{theorem}{Theorem}[section]
\newtheorem{lemma}[theorem]{Lemma}

\newtheorem{proposition}[theorem]{Proposition}
\renewcommand{\theenumi}{\roman{enumi}}
\renewcommand{\labelenumi}{(\theenumi)}

\theoremstyle{definition}
\newtheorem{definition}[theorem]{Definition}
\newtheorem{example}[theorem]{Example}

\theoremstyle{remark}
\newtheorem{remark}[theorem]{Remark}

\renewcommand {\theenumi} {\rm\roman{enumi}}

\usepackage{color}
\usepackage[colorinlistoftodos]{todonotes}
\newcommand{\al}{\alpha}
\newcommand{\be}{\beta}
\newcommand{\ga}{\gamma}

\newcommand{\de}{\delta}

\newcommand{\bx}{\bar x}

\newcommand {\R} {\mathbb R}
\newcommand {\N} {\mathbb N}

\newcommand{\vertiii}[1]{\left\vert\kern-0.25ex\left\vert\kern-0.25ex\left\vert #1\right\vert\kern-0.25ex\right\vert\kern-0.25ex\right\vert}
\newcommand{\vertiiiBig}[1]{\Big\vert\kern-0.25ex\Big\vert\kern-0.25ex\Big\vert #1\Big\vert\kern-0.25ex\Big\vert\kern-0.25ex\Big\vert}

\newcommand{\abs}[1]{\left\vert#1\right\vert}

\newcounter{mycount}


\begin{document}
	
	%-------------------------------------------------------------------------
	% editorial commands: to be inserted by the editorial office
	%
	%\firstpage{1} \volume{228} \Copyrightyear{2004} \DOI{003-0001}
	%\seriesextraline{This is the Concrete Title of this Book\br H.E. R and S.T.C. W, Eds.}
	%\firstpage{1}
	
\title{Metric Constructions and Fixed Point Theorems in Product Spaces}

%----------Author 1
\author{Doan Huu Hieu}

\address{%
Department of Mathematics\\
College of Natural Sciences, Can Tho University\\
Can Tho City,
Vietnam}

\email{dhhieu3107@gmail.com}

%\thanks{This work was completed with the support of our	\TeX-pert.}
%----------Author 2
\author{Vo Minh Tam}
\address{Department of Mathematics\br
Dong Thap University\br
 Dong Thap Province,Vietnam}
\email{vmtam@dthu.edu.vn}

\author{Nguyen Duy Cuong}

\address{%
Department of Mathematics\\
College of Natural Sciences, Can Tho University\\
Can Tho City,
Vietnam}

\email{ndcuong@ctu.edu.vn}



\date{January 20, 2026}

\begin{abstract}
The paper studies a general scheme for constructing metrics on a product of metric spaces by means of a family of continuous convex functions. 
This construction includes the conventional $p$-metrics and generates metrics  that are topologically equivalent to the conventional ones. 
As an application, we study fixed point and approximate fixed point properties for nonexpansive maps on a product space equipped with the constructed metric.
We show that existing fixed point results of this type are consequences of our framework.
Examples are provided to illustrate the established results.
The construction machinery is also used to study  products of length  and geodesic spaces.  
The obtained results encompass existing ones and provide a background for potential studies of fixed point properties on these product spaces.
\end{abstract}

\maketitle

%\setcounter{tocdepth}{2}
%\tableofcontents
\section{Introduction}\label{S1}

Metric fixed point theory is a cornerstone of nonlinear analysis with diverse applications in optimization, game theory, equilibrium problems, variational inequalities, differential equations, and many other areas of pure and applied mathematics. 
The theory has its roots in Poincaré’s work on dynamical systems \cite{Pon86} and is subsequently developed through the fundamental contribution of Banach \cite{Ban22} along with the contributions of Brouwer \cite{Bro12} and Schauder \cite{Sch30}.
The metric space structures play a crucial role in the  theory since notions such as distance, completeness, convergence and others are fundamental for the formulation and analysis of fixed point results \cite{AnsSah23,GolAgaKum21,KhaiKir01}.

A metric space is said to have the \textit{fixed point property} for property $P$  if every self-map possessing property $P$ admits a fixed point.
Given a self-map on a product of two metric spaces, a fundamental problem is to determine which properties of the map together with fixed point properties of the component spaces guarantee the existence of a fixed point.
This problem has been intensively studied in the literature; see, for instance,
\cite{Nad68,For82,KirSte84,TanXu91,EspKir01,Kuc90} and the references therein.
In some situations, however, the existence of an exact fixed point cannot be guaranteed.
This leads to the study of approximate fixed point notions
\cite{KohLeu07,BraMorScaTij03}.
One such notion is an \textit{approximate fixed point sequence} requiring the existence of a sequence whose fixed point error converges to zero \cite{EspKir01}.
A metric space is said to have the \textit{approximate fixed point property} for a property $P$ if every self-map possessing property $P$ admits an approximate fixed point sequence.

In \cite{EspKir01}, among other things, the authors show that a nonexpansive self-map on a product of two metric spaces admits a fixed point when the component spaces have suitable fixed point properties.
Moreover, if one space has the approximate fixed point property for nonexpansive maps and the other is a Banach space, the authors prove that every nonexpansive self-map admits an approximate fixed point sequence.
It is worth mentioning that these results are obtained under the assumption that the product space is equipped with the maximum metric.
It is natural to ask whether analogues of the aforementioned results continue to hold when the product space is endowed with a general metric.
In this paper,  we address this question by establishing fixed point and approximate fixed point results for product spaces equipped with a broad class of metrics. 
In particular, we show that  existing results of this type arise as special cases of our framework.
Examples are provided to illustrate the established results.

Given a finite family of metric spaces, a conventional approach to constructing a metric on the product space is to equip the product space with a $p$-metric for some $p \in [1,\infty]$. 
In this setting, the sum and maximum metrics correspond to $p=1$ and $p=\infty$, respectively. 
More generally, one can construct a product metric by employing an arbitrary norm on $\mathbb{R}^n$ satisfying monotonicity conditions with respect to the coordinatewise order \cite{LooSte90,ChaPet91,BauStoWit61}.
Such monotonicity assumptions are imposed to ensure that the resulting function satisfies the triangle inequality. 
A limitation of this approach, however, is that the class of norms fulfilling these monotonicity conditions is not explicitly known.

In this paper, we propose an alternative approach to constructing metrics on product spaces.
Our construction scheme is based on a class of continuous convex functions  on the standard simplex of $\mathbb{R}^n$.
Given a finite family of sets together with functions defined on them, we construct a new function on the product set in terms of the component functions via a continuous convex function.
It is shown that the resulting function is a metric on the product space if and only if each component function is a metric on the corresponding individual set.
In particular, the $p$-metrics are obtained by plugging suitable convex functions into the constructed formula.
We show that any metric obtained from this construction is topologically equivalent to the conventional ones.

A metric space is called a length space if the distance between any two points equals the greatest lower bound of the lengths of rectifiable curves joining them, and a geodesic space if this greatest lower bound is always attained.
While a general metric space may lack a pathwise interpretation of distance, length and geodesic spaces admit a well-defined notion of motion between points.
Hence, they provide a convenient setting for the study of shortest paths, convexity, curvature and other properties \cite{Pap14}.
To the best of our knowledge, however, existing product constructions for length and geodesic spaces are available only  for $p$-metrics with $p \in [1,\infty]$; see, for instance, \cite{BriHae99}.

The study of fixed point results in length and geodesic spaces has been an active area of research; see
\cite{KohLeu07,Kir04.2,AriLiLop14,KirSha17,KirSha14,AriLeuLop14,ReiZas24} and the references therein.
In this paper, we apply a metric construction scheme for product spaces to the setting of length and geodesic spaces.
More precisely, given a finite family of metric spaces,
we show that the product space equipped with the constructed metric is a length space (a geodesic space)  if and only if each component  is a length space (a geodesic space).
The construction, among other things, provides a  framework for potential studies of fixed point results on a product of  length and geodesic spaces equipped with the metrics.

We would like to note that, among other applications, length and geodesic spaces provide a convenient setting for the study of error bounds, metric regularity and transversality properties in variational analysis \cite{Iof17}.
In complete metric spaces, slope-based sufficient conditions for these properties are typically obtained via the Ekeland variational principle \cite{AzeCorLuc02}, whereas the corresponding necessary conditions usually rely on certain convexity assumptions imposed on the associated functions, sets and set-valued maps \cite{CuoKru22,CuoKru21.3}.
However, in the setting of length spaces, these convexity assumptions can be omitted \cite{Iof01+,AzeCor04}.
Thus, the product construction for length and geodesic spaces developed in this paper not only provides a framework for potential studies of fixed point results in products of such spaces but also contributes to a deeper understanding of the analysis of the aforementioned properties.

The next Section~\ref{S2} contains basic definitions and preliminary results that are used throughout the paper.
Section~\ref{S3} presents a general scheme for constructing metrics on products of metric spaces.
The proposed construction recovers the conventional $p$-metrics and generates metrics that are topologically equivalent to the conventional ones.
In Sections~\ref{S4} and~\ref{S5}, we study the existence of fixed points and approximate fixed point sequences, respectively, for nonexpansive maps on a product of metric spaces equipped with the constructed distance. 
We demonstrate that existing results of this type can be derived as consequences of our general framework.
Examples are given to illustrate the established results.
Sections~\ref{S6} and~\ref{S7} study product constructions for length and geodesic spaces.
In particular, we provide a formula for computing the length of a curve in a product space in terms of the lengths of its component curves.
The obtained results encompass existing ones and form a basis for potential investigations of fixed point results on products of length and geodesic spaces equipped with the general metrics.

\section{Preliminaries}\label{S2}
The following statements recall some basic concepts needed for our study; see, for instance, \cite{EspKir01,AnsSah23}.
\begin{definition}\label{D5.1}
Let  $f: X \to Y$ be a map between metric spaces. 
The map $f$ is
\begin{enumerate}
\item \label{D5.1-1} 
Lipschitz continuous with constant $\theta>0$ if
\begin{gather}\label{lipschitz}
d(f(x),f(x')) \le \theta\cdot d(x,x')\;\;\text{for all}\;\; x,x' \in X;
\end{gather}
\item\label{D5.1-2}
{nonexpansive} if \eqref{lipschitz} holds with $\theta=1$;
\item\label{D5.1-4}
{strictly contractive} if $d(f(x),f(x'))<d(x,x')$ for all $x,x'\in X$ with $x\ne x'$.
\end{enumerate}
\end{definition}

\begin{definition}\label{D5.2}
Let $X$ be a metric space, and  $f: X \to X$. 
\begin{enumerate}
\item\label{D5.2-1}
A point $\bx\in X$ is called a {fixed point} of $f$ if $\bx=f(\bx)$. 
The  set of all fixed points of $f$ is denoted by $\text{Fix}f$.
\item\label{D5.2-2}
A sequence $\{x_n\}\subset X$ is called an {approximate fixed point sequence} of $f$ if $d(x_n, f(x_n))\to 0$ as $n\to\infty$.
\end{enumerate}
\end{definition}

\begin{remark}\label{R5.3}
\begin{enumerate}
\item\label{R5.3-1}
It follows from Definition \ref{D5.1} that 
\begin{center}
strict contractivity $\subset$ nonexpansiveness $\subset$ Lipschitz continuity.
\end{center}
If a strictly contractive map has a fixed point, then this point is unique.
\item\label{R5.3-2}
By Definition~\ref{D5.2}, the existence of a fixed point immediately implies the existence of an approximate fixed point sequence.
The converse implication does not hold in general.
Indeed, let $X:=(0,1)$ be endowed with the metric induced by the norm $|\cdot|$ on $\R$.
The function $f(x) := \tfrac{x}{2}$ admits no fixed point in $X$.
Nevertheless, the sequence
$x_n:= \tfrac{1}{n}$ $(n \in\N)$ is an approximate fixed point sequence of $f$.
\end{enumerate}
\end{remark}

The following definitions recall basic notions from metric geometry needed for the subsequent study; see \cite{BriHae99,Pap14}.
\begin{definition}
Let $(X,d)$ be a metric space, and $a,b\in\R$ with $a\le b$. 
A continuous map $\sigma:[a,b]\to X$ is called a {curve} (or a {path}) in $X$.
If $\sigma(a)=x$ and  $\sigma(b)=y$, then $\sigma$ is said to join $x$ and $y$.
\end{definition}	

\begin{definition}\label{D4.1-1}
Let $(X,d)$ be a metric space, and $a,b\in\R$ with $a\le b$. 
A curve $\sigma:[a,b]\to X$ is
\begin{enumerate}
\item\label{D4.1-1.1}
a {geodesic} if
\begin{gather*}
d(\sigma(t),\sigma(t'))=|t-t'|\;\;\text{for all}\;\;t,t'\in[a,b];
\end{gather*}	
\item\label{D4.1-1.2}
a {constant speed geodesic} with number $\lambda\ge 0$ if
\begin{gather*}
d(\sigma(t),\sigma(t'))=\lambda\cdot|t-t'|\;\;\text{for all}\;\;t,t'\in[a,b].
\end{gather*}	 
\end{enumerate}
\end{definition}	

The following statement gives a way to construct a geodesic via a constant speed geodesic.
\begin{proposition}\label{P4.2}
Let $(X,d)$ be a metric space, and $a,b\in\R$ with $a\le b$. 
If a curve $\sigma:[a,b]\to X$ is a \textit{constant speed geodesic} with number $\lambda>0$, then the curve
\begin{gather}\label{R4.5-4} \gamma(t):=\sigma\left(a+\frac{t}{\lambda}\right)\;\;\text{for all }\;\;t\in[0,\lambda(b-a)]
\end{gather}
is a \textit{geodesic}.
\end{proposition}

\begin{proof}
Let $t,t'\in [0,\lambda(b-a)]$.
Then $a+\frac{t}{\lambda},a+\frac{t'}{\lambda}\in[a,b]$.
By Definition~\ref{D4.1-1}\eqref{D4.1-1.2},
\begin{align*}
d(\gamma(t),\gamma(t'))
=d\left(\sigma\left(a+\frac{t}{\lambda}\right),\sigma\left(a+\frac{t'}{\lambda}\right)\right)
=\lambda\cdot \left|\left(a+\frac{t}{\lambda}\right)- \left(a+\frac{t'}{\lambda}\right)  \right|=|t-t'|.
\end{align*}		
By Definition~\ref{D4.1-1}\eqref{D4.1-1.1}, $\gamma$ is a geodesic.
\end{proof}	

\section{Product Metric Constructions}\label{S3}
In this section, we present a general scheme for constructing metrics on a product of metric spaces.
Let $\Omega_n:=\{(t_1,\ldots,t_{n})\in\R^{n}_+\mid \sum_{i=1}^nt_i =1\}$, where $\R_+$ denotes the set of nonnegative real numbers and $n\ge 2$.
Observe that $\max\{t_1,\ldots,t_n\}\ge \frac{1}{n}$ for all $(t_1,\ldots,t_{n})\in\Omega_n$.
The set $\Omega_n$ is a convex and compact subset of $\R^n$.
Denote by $\pmb{\Psi}_n$ the class of all continuous convex functions $\psi:\Omega_n\to\R$ satisfying
$\psi(\mathbf{e}_1)=\cdots=\psi(\mathbf{e}_n)=1$, where $\mathbf{e}_1,\ldots,\mathbf{e}_n$ are the standard basis vectors of $\R^n$, and
\begin{gather*}
\psi(t_1,\ldots,t_n)\ge (1-t_i)\cdot\psi\left(\dfrac{t_1}{1-t_i},\ldots,\dfrac{t_{i-1}}{1-t_i},0,\dfrac{t_{i+1}}{1-t_i},\ldots,\dfrac{t_{n}}{1-t_i}\right)
\end{gather*}	
for all $(t_1,\ldots,t_n)\in\Omega_n$ with $t_i<1$ 
$(i=1,\ldots,n)$.
This class of functions has been studied in \cite{SaiKatTak00,Cuo25} for the construction of norms on a product of normed spaces.

\begin{remark}\label{R2.1}
\begin{enumerate}
\item
In this paper, we write $\infty$ instead of $+\infty$.
For a scalar $p\in[1,\infty]$, define
\begin{gather}\label{ppsi}
\psi_{p}(t_1,\ldots,t_n)
:=\begin{cases}
\left(t_1^p+\cdots+t_n^p\right)^{\frac{1}{p}}  & \text{if } p\in[1,\infty),\\
\max\{t_1,\ldots,t_n\} & \text{if } p=\infty
\end{cases} 
\end{gather}
for all $(t_1,\ldots,t_n)\in\Omega_n$.
One can verify that $\psi_p\in \pmb{\Psi}_n$; see \cite{Cuo26} for a detailed proof.
%Moreover, $\psi_1(t_1,\ldots,t_n)=1$ for all $(t_1,\ldots,t_n)\in\Omega_n$.
\item\label{R2.1-2}
Let $\psi\in \pmb{\Psi}_n$.
Then 
\begin{gather}\label{R2.1-4}
\max\{t_1,\ldots,t_n\}\le\psi(t_1,\ldots,t_n)\le 1
\end{gather}	
for all $(t_1,\ldots,t_n)\in\Omega_n$; see \cite{Cuo25}.
\end{enumerate}
\end{remark}

The following statement establishes a monotonicity-type property of the generating function; see \cite{Cuo25}.
This property plays a key role in proving the triangle inequality for the constructed function on the product set.
\begin{lemma}\label{L2.2}
Let $\psi\in\pmb{\Psi}_n$, $0<\al_i\le\be_i$ $(i=1,\ldots,n)$, $\al:=\sum_{i=1}^{n}\al_i$ and 
$\be:=\sum_{i=1}^{n}\be_i$.
Then
\begin{gather}\label{L2.3-1}
\al \cdot\psi\left(\dfrac{\al_1}{\al},\ldots,\dfrac{\al_{n}}{\al}\right)\le \be \cdot\psi\left(\dfrac{\be_1}{\be},\ldots,\dfrac{\be_{n}}{\be}\right).
\end{gather}		
\end{lemma}

The following theorem characterizes the metric structure of a product space in terms of its components.
\begin{theorem}\label{T3.10}
Let $X_i$ be a nonempty set, $d_i:X_i\times X_i\to\R$ $(i=1,\ldots,n)$, $M:=X_1\times\cdots\times X_n$, and $\psi\in\pmb{\Psi}_n$.
Define 
\begin{equation}\label{T2.2-1}
d_\psi(x,y)
:= \begin{cases}
\left(\sum_{i=1}^n d_i(x_i,y_i)\right)\cdot\psi\left(\dfrac{d_1(x_1,y_1)}{\sum_{i=1}^nd_i(x_i,y_i)},\ldots,\dfrac{d_n(x_n,y_n)}{\sum_{i=1}^nd_i(x_i,y_i)}\right)  & \text{\rm if } x\ne y,\\
0 & \text{\rm if } x= y
\end{cases} 
\end{equation}
for all $x:=(x_1,\ldots,x_n), y:=(y_1,\ldots,y_n)\in M$.
Then $d_i$ is a metric on $X_i$ $(i=1,\ldots,n)$ if and only if $d_\psi$ is a metric on $M$.
\end{theorem}	

\begin{proof}
Suppose that $d_i$ is a metric on $X_i$ $(i=1,\ldots,n)$.
To prove that $d_\psi$ is a metric on $M$, it suffices to verify the triangle inequality since the other conditions hold trivially.
Let $x:=(x_1,\ldots,x_n),y:=(y_1,\ldots,y_n),z:=(z_1,\ldots,z_n)\in M$.
If $x=y$, $x=z$, or $y=z$, then
\begin{gather}\label{T2.2-2}
d_\psi(x,z)\le d_\psi(x,y)+d_\psi(y,z)
\end{gather}
holds as equality.
Suppose that $x,y,z$ are distinct.
Let
\begin{gather}\notag
\lambda_i:=d_i(x_i,z_i),\;
\vartheta_i:=d_i(x_i,y_i),\; \xi_i:=d_i(y_i,z_i)\;\;(i=1,\ldots,n),\\	\label{T2.2-6}
\lambda:=\sum_{i=1}^{n}\lambda_i,
\vartheta:=\sum_{i=1}^{n}\vartheta_i,\;\xi:=\sum_{i=1}^{n}\xi_i,\; t:=\left(\frac{\vartheta_1}{\vartheta},\ldots,\frac{\vartheta_n}{\vartheta}\right),\;t':=\left(\frac{\xi_1}{\xi},\ldots,\frac{\xi_n}{\xi}\right).
\end{gather}
By the triangle inequality,
\begin{gather}\label{T2.2-3}
\lambda_i\le \vartheta_i+\xi_i\;\;(i=1,\ldots,n),\;\;
\lambda\le \vartheta+\xi.
\end{gather}	
Then
\begin{eqnarray*}
d_\psi(x,z)
&\overset{ \eqref{T2.2-1}}{=}&\lambda\cdot\psi\left(\dfrac{
\lambda_1}{\lambda},\ldots,\dfrac{\lambda_n}{\lambda}\right)\\ 
&\overset{\eqref{L2.3-1},\eqref{T2.2-3}}{\le}&
(\vartheta+\xi)\cdot\psi\left(\dfrac{
\vartheta_1+\xi_1}{\vartheta+\xi},\ldots,\dfrac{
\vartheta_n+\xi_n}{\vartheta+\xi}\right)\\ 
&\overset{\eqref{T2.2-6}}{=}&(\vartheta+\xi)\cdot\psi\left(\dfrac{\vartheta}{\vartheta+\xi}\cdot t+\dfrac{\xi}{\vartheta+\xi}\cdot t'\right)\\
&\le&(\vartheta+\xi)\cdot\left(\dfrac{\vartheta}{\vartheta+\xi}\cdot\psi(t)+ \dfrac{\xi}{\vartheta+\xi}\cdot\psi(t') \right)\;\;\text{($\psi$ is convex)}\\
&=&\vartheta\cdot\psi(t)+\xi\cdot\psi(t')\\
&\overset{ \eqref{T2.2-1}}{=}& d_\psi(x,y)+d_\psi(y,z).
\end{eqnarray*}	
Thus, $d_\psi$ is a metric on $M$.

Suppose that $d_\psi$ is a metric on $M$.
Fix $i=1,\ldots,n$ and $x_j\in X_j$ $(j\ne i)$.
Let $x_i,y_i\in X_i$ and 
\begin{gather}\label{T4.8-9}
x:=(x_1,\ldots,x_{i-1},x_i,x_{i+1},\ldots,x_n),\;\;
y:=(x_1,\ldots,x_{i-1},y_i,x_{i+1},\ldots,x_n).
\end{gather}	
By \eqref{T2.2-1}, 
$d_\psi(x,y)=d_i(x_i,y_i)\cdot\psi(\mathbf{e}_1)=d_i(x_i,y_i)$.
Since $d_\psi$ is symmetric and nonnegative, the same properties hold for $d_i$.
If $x_i=y_i$, then $d_\psi(x,y)=0$ since $x=y$, and consequently, $d_i(x_i,y_i)=0$.
Conversely, if $d_i(x_i,y_i)=0$, then $x=y$ since $d_\psi(x,y)=0$, and consequently, $x_i=y_i$.
Let $z_i\in X_i$ and  $z:=(x_1,\ldots,x_{i-1},x_i,x_{i+1},\ldots,x_n)$.
The fact
$d_\psi(x,z)\le d_\psi(x,y)+d_\psi(y,z)$ implies $d_i(x_i,z_i)\le d_i(x_i,y_i)+d_i(y_i,z_i)$.
Hence, $d_i$ is a metric on $X_i$.
\end{proof}	

\begin{remark}\label{R2.4}
\begin{enumerate}
\item\label{R2.4-1}
Let $\psi:=\psi_p$ be given by \eqref{ppsi} with $p\in[1,\infty]$, and $(X_i,d_i)$ $(i=1,\ldots,n)$ be metric spaces.
By Theorem~\ref{T3.10}, we obtain the $p$-metrics on $M$:
\begin{gather}\label{pmetric}
d_{\psi_p}(x,y)
= \begin{cases}
(d^p_1(x_1,y_1)+\cdots+d^p_n(x_n,y_n))^{\frac{1}{p}} & \text{if } p\in[1,\infty),\\
\max\{d_1(x_1,y_1),\ldots,d_n(x_n,y_n)\} & \text{if } p=\infty
\end{cases} 
\end{gather}
for all $x:=(x_1,\ldots,x_n), y:=(y_1,\ldots,y_n)\in M$.
\item\label{R2.4-2}
Let $\psi\in\pmb{\Psi}_n$.
By Remark~\ref{R2.1}\eqref{R2.1-2} and Theorem~\ref{T3.10},  
\begin{gather}\label{R2.4-3}
\max_{1\le i\le n} d_i(x_i,y_i) \le d_\psi(x,y)\le\sum_{i=1}^{n}d_i(x_i,y_i)
\end{gather}	
 for all $x:=(x_1,\ldots,x_n),y:=(y_1,\ldots,y_n)\in M$.
Thus, the constructed metrics are topologically equivalent to the conventional $p$-metrics.
\item 
The direct implication of Theorem~\ref{T3.10} seems to strengthen
 \cite[Definition~3]{AvgFon00}
since it encompasses the maximum metric
which is not covered there.
\end{enumerate}
\end{remark}	

\begin{lemma}\label{L3.5}
Let $\psi\in\pmb{\Psi}_2$ and
\begin{gather}\label{T5.1.0}
\lambda:=2\psi\left(\frac{1}{2},\frac{1}{2}\right),\;\;\theta:=(1+\lambda)\psi\left(\frac{\lambda}{\lambda+1},\frac{1}{\lambda+1}\right).
\end{gather}
Then $1\le\lambda\le 2$ and $\lambda\le \theta\le 1+\lambda$.
\end{lemma}	

\begin{proof}
By \eqref{R2.1-4}, we have $1\le\lambda\le 2$, and
\begin{gather*}
\lambda=(1+\lambda)\max\left\{\frac{\lambda}{\lambda+1},\frac{1}{\lambda+1}\right\}\le
\theta\le (1+\lambda)\left(\dfrac{\lambda}{\lambda+1}+\dfrac{1}{\lambda+1}\right)=1+\lambda.
\end{gather*}		
This completes the proof.
\end{proof}	

\section{Fixed Points Theorems in Product Spaces}\label{S4}
Let $(X_1,d_1)$ and $(X_2,d_2)$ be metric spaces, and $X:=X_1\times X_2$ be equipped with the metric $d_\psi$ given by  \eqref{T2.2-1} for some $\psi\in\pmb{\Psi}_2$.
Define $T:X\to X$ by 
\begin{align}\label{T}
	T(x_1,x_2):=(T_1(x_1,x_2),T_2(x_1,x_2))\;\;\text{for all}\;\;(x_1,x_2)\in X,
\end{align}	
where $T_1:X\to X_1$ and $T_2:X\to X_2$. 

The following theorem proves the existence of fixed points for nonexpansive mappings on the product space.
\begin{theorem}\label{T5.1}
Let $\lambda$ and $\theta$ be given by \eqref{T5.1.0}.
Suppose that the following conditions are satisfied:
\begin{enumerate}
\item\label{T5.1-1}
$X_1$ has the fixed point property for strictly contractive maps;
\item\label{T5.1-2}
$X_2$ has the fixed point property for  Lipschitz continuous maps with constant $\theta$;
\item\label{T5.1-3}
$T$ is a nonexpansive map satisfying
\begin{gather}\label{T5.1.1}
d_1(T_1(x_1,x_2),T_1(x'_1,x'_2))<\dfrac{\lambda}{\theta}\cdot d_\psi((x_1,x_2),(x'_1,x'_2))
\end{gather}	
for all $(x_1,x_2),(x'_1,x'_2)\in X$ with $d_1(x_1,x'_1)>\lambda d_2(x_2,x'_2)$.
\end{enumerate}	
Then $T$ has a fixed point.
\end{theorem}	

\begin{proof}
\textbf{Claim 1.}
For each $x_2\in X_2$,
the map $T_{x_2}:X_1\to X_1$ given by
\begin{gather}\label{def-Tz}
T_{x_2}(x_1):=T_1(x_1,x_2)\;\;\text{for all}\;\; x_1\in X_1
\end{gather}
has a unique fixed point $\xi(x_2)\in X_1$.
Indeed, let $x_1,x'_1\in X_1$ with $x_1\ne x'_1$.
Then $d_1(x_1,x'_1)>0$, and consequently,
\begin{eqnarray*}
d_1(T_{x_2}(x_1), T_{x_2}(x'_1))
&\overset{\eqref{def-Tz}}{=}&
d_1(T_1(x_1,x_2),T_1(x'_1,x_2))\\
&\overset{\eqref{T5.1.1}}{<}&
\dfrac{\lambda}{\theta} d_\psi((x_1,x_2),(x'_1,x_2))\\
&\overset{\eqref{T2.2-1}}{=}&
\dfrac{\lambda}{\theta} d_1(x_1,x'_1)\psi(1,0)\le d_1(x_1,x'_1),
\end{eqnarray*}
where the last inequality follows from the fact that $\psi(1,0)=1$ and 
$\lambda/\theta\le 1$ (Lemma~\ref{L3.5}).
Thus, $T_{x_2}$ is strictly contractive on $X_1$.
By assumption \eqref{T5.1-1}, it has a unique fixed point $\xi(x_2)\in X_1$.\\
\textbf{Claim 2.}
For any $x_2,x_2'\in X_2$, 
\begin{gather}\label{T5.4.3}
A_1\le \lambda\cdot A_2,
\end{gather}
where 
\begin{gather}\label{T5.4.10}
A_1:=d_1(\xi(x_2),\xi(x_2'))\;\;\text{and}\;\;A_2:=d_2(x_2,x_2').
\end{gather}
Indeed, let $\xi(x_2)$ and $\xi(x_2')$ be the corresponding  fixed points of the maps $T_{x_2}$ and $T_{x_2'}$.
Then
\begin{gather}\label{T5.4.4}
\xi(x_2)=T_1(\xi(x_2),x_2),\;\;\xi(x_2')=T_1(\xi(x_2'),x_2').
\end{gather}
Suppose that $A_1>\lambda A_2$.
Then 
\begin{eqnarray}
A_1
&\overset{\eqref{T5.4.4}}{=}&
d_1(T_1(\xi(x_2),x_2),T_1(\xi(x'_2),x'_2))\notag\\ 
&\overset{\eqref{T5.1.1}}{<}&
\dfrac{\lambda}{\theta}
d_\psi((\xi(x_2),x_2),(\xi(x'_2),x'_2)) \notag\\
&\overset{\eqref{T2.2-1}}{=}&
\dfrac{\lambda}{\theta}(A_1+A_2)\psi\left(\dfrac{A_1}{A_1+A_2},\dfrac{A_2}{A_1+A_2}\right) \notag\\
&\overset{\eqref{L2.3-1}}{\le}&
\dfrac{\lambda}{\theta}
\left(1+\dfrac{1}{\lambda}\right)\psi\left(\dfrac{A_1}{A_1+A_1/\lambda},\dfrac{A_1/\lambda}{A_1+A_1/\lambda} \right)\cdot A_1\notag\\ 
&=& \dfrac{\lambda}{\theta}\left(1+\dfrac{1}{\lambda}\right)\psi\left(\dfrac{\lambda}{\lambda+1},\dfrac{1}{\lambda+1}\right)\cdot A_1
= A_1,\label{T5.1.2}
\end{eqnarray}	
a contradiction.\\
\textbf{Claim 3.}
The map $G: X_2\to X_2$ given by
\begin{gather}\label{G}
G(x_2)=T_2(\xi(x_2),x_2)\;\;\text{for all}\;\;x_2\in X_2
\end{gather}	 
is Lipschitz continuous with constant $\theta$.
Indeed, for any $x_2,x'_2\in X_2$, 
\begin{eqnarray*}
d_2(G(x_2),G(x_2'))
&\overset{\eqref{G}}{=}&
d_2(T_2(\xi(x_2),x_2),T_2(\xi(x'_2),x'_2))\\
&\overset{\eqref{R2.4-3}}{\le}& 
d_\psi(T(\xi(x_2),x_2),T(\xi(x'_2),x'_2))\\
&\le&
d_\psi((\xi(x_2),x_2),(\xi(x_2'),x_2'))\;\;\;\;(T\;\text{is nonexpansive})\\
&\overset{\eqref{T2.2-1}}{=}& (A_1+A_2)\psi\left(\dfrac{A_1}{A_1+A_2},\dfrac{A_2}{A_1+A_2}\right)\\
&\overset{\eqref{L2.3-1},\eqref{T5.4.3}}{\le}&
(1+\lambda)\psi\left(\dfrac{\lambda A_2}{\lambda A_2+A_2},\dfrac{A_2}{\lambda A_2+A_2}\right) \cdot A_2\\
&=&(1+\lambda)\psi\left(\dfrac{\lambda}{\lambda+1},\dfrac{1}{\lambda+1}\right) \cdot A_2\\
&=&\theta d_2(x_2,x_2').
\end{eqnarray*}	
\textbf{Claim 4.}
There exists an $\bx_2\in X_2$ such that $(\xi(\bx_2),\bx_2)=T(\xi(\bx_2),\bx_2)$.
Indeed, by {Claim 3} and assumption \eqref{T5.1-2}, there exists a point $\bx_2 \in X_2$ such that
$\bar x_2=T_2(\xi(\bar x_2),\bar x_2).$
In view of the property of $\xi(\bx_2)$, we have $\xi(\bx_2)=T_1(\xi(\bar x_2),\bar x_2)$.
This completes the proof.
\end{proof}	

\begin{remark}\label{R5.5}
\begin{enumerate}
\item
In general, a nonexpansive map does not necessarily admit a fixed point. 
However, the existence of fixed points can be ensured under additional assumptions on the underlying space or the map; see, for instance, \cite[Theorem~5]{Bai88}.  
Similarly, fixed points for Lipschitz maps can also be guaranteed if appropriate assumptions are imposed; see \cite[Theorem~2.7]{Bis23}, \cite[Lemma~3.1]{Gor96} and \cite[Theorem~6]{Gor01}.
\item 
Theorem~\ref{T5.1} studies fixed point properties for nonexpansive maps on a product of metric spaces endowed with the general metric.
Fixed point results on a product of Banach spaces equipped with $p$-norms $(p\in[1,\infty])$ under different assumptions can be found in \cite[Theorems 2.1, 2.2 \& 2.3]{YanKir88} and \cite[Theorems 2.1 \& 2.2]{TanXu91}.
\item	
By Lemma~\ref{L3.5}, $\frac{\lambda}{\theta}\le 1$ for any $\psi\in\pmb{\Psi}_2$.
Let $\Omega_2:=\{(t_1,t_2)\in\R^2_+ \mid t_1+t_2=1\}$ and consider the function $\psi_\infty$ given by \eqref{ppsi}, i.e.,
\begin{gather}\label{R5.5-1}
\psi_\infty(t_1,t_2):=\max\{t_1,t_2\}\;\;\text{for all}\;\;(t_1,t_2)\in\Omega_2.
\end{gather}	
By \eqref{pmetric},
\begin{gather}\label{R5.5-2}
d_{\psi_\infty}((x_1,x_2),(x'_1,x'_2))=\max\{d_1(x_1,x'_1),d_2(x_2,x'_2)\}
\end{gather}
for all $(x_1,x_2),(x'_1,x'_2)\in X_1\times X_2.$
By \eqref{T5.1.0} and \eqref{R5.5-1}, we have $\lambda=\theta=1$.
In this case, the Lipschitz continuity with constant $\theta$ reduces to the nonexpansiveness, and consequently, Theorem~\ref{T5.1} strengthens \cite[Theorem~3.1]{EspKir01} since condition \eqref{T5.1.1} is required to hold for all pairs of points satisfying
$d_1(x_1,x'_1)>d_2(x_2,x'_2)$, while in \cite[Theorem~3.1]{EspKir01} the corresponding condition is imposed for all pairs of points with $d_1(x_1,x'_1)\ne d_2(x_2,x'_2)$.
\end{enumerate}
\end{remark}	

The next example illustrates Theorem~\ref{T5.1}.
\begin{example}\label{E4.3}
Let $X_1=X_2:=[-1,1]$ be  equipped with the metric induced by the norm $|\cdot|$ on $\R$, and $X:=X_1\times X_2$ be equipped with the metric $d_{\psi_p}$ where
$\psi_p$ is given by \eqref{ppsi} for some $p\in[1,\infty)$.
Then
\begin{gather}\label{E5.6-1}
\lambda:=2\psi_p\left(\frac{1}{2},\frac{1}{2}\right)=2^{\frac{1}{p}},\;\;
\theta:=(1+\lambda)\psi_p\left(\frac{\lambda}{\lambda+1},\frac{1}{\lambda+1}\right)=3^\frac{1}{p}.
\end{gather}
Let $\alpha:=\frac{1}{2}\left(\frac{2}{3}\right)^{\frac{1}{p}}$, $\beta:=\frac{1}{2}\left(\frac{4}{3}\right)^{\frac{1}{p}}$ and $\gamma:=\left(\frac{1}{3}\right)^{\frac{1}{p}}.$
Define $T: X\to X$ by 
\begin{equation*}
T(x):=(T_1(x),T_2(x)):=(\alpha x_1+\beta x_2, \gamma x_2)\;\;\text{for all}\;\;x:=(x_1,x_2)\in X.
\end{equation*}	
Since $p\geq 1$, the function $f(t):=t^{\frac{1}{p}}$ for all $t\ge 0$ is concave, and consequently,
\begin{equation*}
\al+\be=\frac{f(2/3)+f(4/3)}{2}\leq f\left(\frac{2/3+4/3}{2}\right)=f(1)=1.
\end{equation*}
Then $|T_1(x_1,x_2)|=|\alpha x_1+\beta x_2|\leq \alpha|x_1|+\beta|x_2|\leq \alpha+\beta\leq 1$ and  $|T_2(x_1,x_2)|=|\gamma x_2|=\gamma |x_2|\leq 1$
for all $(x_1,x_2)\in X$.
Thus, the map $T$ is well-defined. 
Observe that 
\begin{gather}\label{E5.6-2}
(a+b)^p\leq 2^{p-1}(a^p+b^p)\;\;\text{for all}\;\;a,b\ge 0.
\end{gather}
Indeed, since $p\ge 1$, the function $g(t):=t^p$ for all $t\ge 0$ is convex.
Thus, 
\begin{gather*}
g\left(\frac{a+b}{2}\right)\le\frac{g(a)+g(b)}{2}\Leftrightarrow \left(\frac{a+b}{2}\right)^p\le\dfrac{a^p+b^p}{2}
\Leftrightarrow (a+b)^p\leq 2^{p-1}(a^p+b^p).
\end{gather*}	
Let $x:=(x_1,x_2),x':=(x'_1,x'_2)\in X$.
Then
\begin{align*}
d_{\psi_p}(T(x), T(x'))&=\left(|\alpha(x_1-x_1')+\beta(x_2-x_2')|^p+|\gamma (x_2-x_2')|^p\right)^{\frac{1}{p}}\\
&\le\left((\alpha|x_1-x_1'|+\beta|x_2-x_2'|)^p+\ga^p| (x_2-x_2')|^p\right)^{\frac{1}{p}}\\
&\leq \left(2^{p-1}\alpha^p|x_1-x_1'|^p+(2^{p-1}\beta^p+\gamma^p)|x_2-x_2'|^p\right)^{\frac{1}{p}}\\
&= \left(\frac{1}{3}|x_1-x_1'|^p+|x_2-x_2'|^p\right)^{\frac{1}{p}}\\
&\leq\left(|x_1-x_1'|^p+|x_2-x_2'|^p\right)^{\frac{1}{p}}
=d_{\psi_p}(x,x').
\end{align*}
Thus, $T$ is nonexpansive.
If $|x_1-x'_1|>\lambda |x_2-x'_2|$, then
\begin{align*}
|T_1(x)-T_1(x')|
&=|\alpha(x_1-x_1')+\beta(x_2-x_2')|\\
&<\left(\alpha+\frac{\beta}{\lambda}\right)|x_1-x'_1|\\
&=\left(\frac{\lambda}{2\theta}+\frac{\lambda}{2\theta}\right)|x_1-x'_1|
\le\dfrac{\lambda}{\theta}\cdot d_{\psi_p}(x,x').
\end{align*}	
Thus, condition \eqref{T5.1.1} is satisfied.
By the intermediate value theorem, $X_1$ and $X_2$ have the fixed point property for every continuous self-maps.
Thus, conditions \eqref{T5.1-1} and \eqref{T5.1-2} in Theorem~\ref{T5.1} are satisfied. 
Thus, $\text{Fix}T\neq\emptyset$. 
Moreover, direct computation shows that $\text{Fix}T=\{(0,0)\}$.
\end{example}

The following theorem provides a parallel fixed point result to Theorem~\ref{T5.1}.
\begin{theorem}\label{T5.2}
Let $\lambda$ and $\theta$ be given by \eqref{T5.1.0}.
Suppose that the following conditions are satisfied:
\begin{enumerate}
\item\label{T5.2-1}
$X_1$ has the fixed point property for nonexpansive maps;
\item\label{T5.2-2}
$X_2$ has the fixed point property for  Lipschitz maps with constant $\theta$;
\item\label{T5.2-3}
$T$ is a nonexpansive map such that condition \eqref{T5.1.1} holds for all $(x_1,x_2),(x'_1,x'_2)\in X$ with $x_1\ne x'_1$ and 
$x_2\ne x'_2$.
\end{enumerate}	
Then $T$ has a fixed point.
\end{theorem}	

\begin{proof}
For each $x_2\in X_2$, the map $T_{x_2}:X_1\to X_1$ given by \eqref{def-Tz} is nonexpansive.
Indeed, for any $x_1,x'_1\in X_1$, we have
\begin{eqnarray*}
d_1(T_{x_2}(x_1), T_{x_2}(x'_1))
&\overset{\eqref{def-Tz}}{=}&
d_1(T_1(x_1,x_2),T_1(x'_1,x_2))\\
&\overset{\eqref{R2.4-3}}{\le}& d_\psi(T(x_1,x_2),T(x'_1,x_2))\\
&\le& d_\psi((x_1,x_2),(x'_1,x_2))\;\;(T\;\text{is nonexpansive})\\
&\overset{\eqref{T2.2-1}}{=}&
d_1(x_1,x'_1).
\end{eqnarray*}
Thus, $T_{x_2}$ is nonexpansive.
By assumption \eqref{T5.2-1}, the set ${\rm{Fix}}T_{x_2}$ is nonempty.
Let $\xi:X_2\to X_1$ be a selection of the set-valued mapping $x_2\mapsto {\rm{Fix}}T_{x_2}$.
Let $x_2,x_2'\in X_2$ and $A_1,A_2$ be given by \eqref{T5.4.10}.
Then condition \eqref{T5.4.3} is satisfied.
Indeed, if $\xi(x_2)=\xi(x_2')$, then \eqref{T5.4.3} holds trivially.
If $\xi(x_2)\ne\xi(x_2')$, then it must hold that $x_2\ne x_2'$.
We have
\begin{eqnarray*}
A_1
&\overset{\eqref{T5.4.4}}{=}&
d_1(T_1(\xi(x_2),x_2),T_1(\xi(x'_2),x'_2))\\
&\overset{\eqref{T5.1.1}}{<}&
\dfrac{\lambda}{\theta}
d_\psi((\xi(x_2),x_2),(\xi(x'_2),x'_2))\\
&\overset{\eqref{T2.2-1}}{=}&
\dfrac{\lambda}{\theta}(A_1+A_2)\psi\left(\dfrac{A_1}{A_1+A_2},\dfrac{A_2}{A_1+A_2}\right).
\end{eqnarray*}	
If $A_1>\lambda A_2$, then as shown in \eqref{T5.1.2}, we obtain the contradiction that $A_1<A_1$.
Let $G:X_2\to X_2$ be defined by \eqref{G}.
By Claim~3 in the proof of Theorem~\ref{T5.1}, the mapping $G$ is Lipschitz continuous with constant $\theta$.  
The remainder of the proof follows the same lines as Claim~4 in the proof of Theorem~\ref{T5.1}.
\end{proof}	

\begin{remark}
\begin{enumerate}	
\item	
When $\psi:=\psi_\infty$ is given by \eqref{R5.5-1}, we have $\lambda = \theta = 1$, and consequently, the Lipschitz continuity  reduces to the nonexpansiveness. 
In this case, Theorem~\ref{T5.2} recaptures \cite[Theorem~3.3]{EspKir01}.
\item 
Theorems~\ref{T5.1} and \ref{T5.2} are  incomparable since the sets of points satisfying condition~\eqref{T5.1.1} in each theorem cannot be compared in general. 
Nevertheless, in view of Remark~\ref{R5.3}\eqref{R5.3-1}, the assumption~\eqref{T5.2-1} in Theorem~\ref{T5.2} is stronger than the assumption~\eqref{T5.1-1} in Theorem~\ref{T5.1}.
\end{enumerate}
\end{remark}	

The following example illustrates  Theorem~\ref{T5.2}.
\begin{example}
Let $X_1$ and $X_2$ be the metric spaces given in Example~\ref{E4.3}, 
$X:=X_1\times X_2$ be equipped with the metric $d_{\psi_p}$ where
$\psi_p$ is given by \eqref{ppsi} for some $p\in[1,\infty)$, $\lambda$ and $\theta$ be given by \eqref{E5.6-1}.
Define $T: X\to X$ by 
\begin{equation*}
T(x_1,x_2):=(T_1(x),T_2(x)):=\left(\frac{\lambda}{2\theta}\cdot x_1, x_2\right)\;\;\text{for all}\;\;(x_1,x_2)\in X.
\end{equation*}	
Note that $\frac{\lambda}{2\theta}=\frac{1}{2}\left(\frac{2}{3}\right)^{\frac{1}{p}}<1$.
The map $T^{\Omega}$ is well-defined since 
\begin{gather*}
|T_1(x_1,x_2)|=\frac{\lambda}{2\theta}|x_1|\le 1\;\;\text{and}\;\;
|T_2(x_1,x_2)|=|x_2|\le 1\;\;\text{for any}\;\;(x_1,x_2)\in X.
\end{gather*}	
Let  $x:=(x_1,x_2),x':=(x'_1,x'_2)\in X$.
Then
\begin{align*}
d_{\psi_p}(T(x), T(x'))=\left(\left|\frac{\lambda}{2\theta}(x_1-x_1')\right|^p+|x_2-x_2'|^p\right)^{\frac{1}{p}}
\le d_{\psi_p}(x,x').
\end{align*}
Thus, $T$ is nonexpansive.
If $x_1\ne x'_1$ and $x_2\ne x'_2$, then 
\begin{align*}
|T_1(x)-T_1(x')|=\left|\frac{\lambda}{2\theta}(x_1-x_1')\right|
<\frac{\lambda}{\theta}|x_1-x_1'|\leq\dfrac{\lambda}{\theta} d_{\psi_p}(x,x').
\end{align*}	
Hence, condition \eqref{T5.1-3} is satisfied. 
By the intermediate value theorem, $X_1$ and $X_2$ have the fixed point property for every continuous self-maps.
Thus, conditions \eqref{T5.2-1} and \eqref{T5.2-2} in Theorem~\ref{T5.2} are satisfied. 
Thus, $\text{Fix}T\neq\emptyset$. 
Moreover, direct computation shows that $\text{Fix}T=\{0\}\times [-1,1]$.
\end{example}

\section{Approximate Fixed Point Theorems in Product Spaces}\label{S5}
%\subsection{Sufficient conditions for the existence of approximate fixed point sequences}
In this section, we study a particular case of the model studied in Section~\ref{S4}.
More precisely, we assume that $(X_1,\|\cdot\|)$ is a normed vector space with the associated metric given by
\begin{gather*}
	d_1(x_1,x'_1):=\|x_1-x'_1\|\;\;\text{for all}\;\;x_1,x'_1\in X_1.
\end{gather*}	
Let $\Omega$ be a nonempty bounded closed convex subset of $X_1$,  
and $T^\Omega:=(T^\Omega_1,T_2)$ be the restriction of the map $T:=(T_1,T_2)$ given by \eqref{T} on $X_{\Omega}:=\Omega\times X_2$.

The following result is needed for our study; see \cite[Theorem~1.1]{BerRus20} and \cite[Remark, p.69]{Shi76}.
\begin{lemma}\label{L5.6}
Let $X_1$ be a Banach space, $Q:\Omega\to \Omega$ be a nonexpansive map, $\omega\in \Omega$, $\al\in(0,1)$, and
$P:\Omega\to \Omega$ given by 
$P:=\al I+(1-\al)Q$ where $I$ is the identity map.
Then
\begin{gather*}
\|P^{n+1}(\omega)-P^n(\omega)\|\to 0\;\;\text{as}\;\; n\to\infty.	
\end{gather*}	
\end{lemma}

The following theorem proves the existence of an approximate fixed point sequence  for nonexpansive maps on the product space.
\begin{theorem}\label{T5.7}
Let $\psi\in\pmb{\Psi}_2$, $\lambda$ be given in \eqref{T5.1.0}, $X_1$ be a Banach space, and $X_2$ be a metric space with the approximate fixed point property for Lipschitz continuous maps with constant $\lambda$.
Suppose that $T^\Omega$ is a nonexpansive map satisfying
\begin{gather}\label{T5.3.2}
\|T_1^\Omega(x_1,x_2)-T_1^\Omega(x'_1,x'_2)\|\le \dfrac{1}{\lambda}\cdot d_\psi((x_1,x_2),(x'_1,x'_2))
\end{gather}	
for all $(x_1,x_2),(x'_1,x'_2)\in X_{\Omega}$ with $\|x_1-x'_1\|\le d_2(x_2,x'_2)$.
Then $T^\Omega$ has an approximate fixed point sequence.
\end{theorem}	

\begin{proof}
For each $x_2\in X_2$, let $T^\Omega_{x_2}:\Omega\to X_1$ be the restriction of
$T_{x_2}:X_1\to X_1$ given by \eqref{def-Tz}. 
Fix $\al\in(0,1)$ and define  $P_{x_2}:\Omega\to \Omega$  by
\begin{gather}\label{T5.3.1}
P_{x_2}:=\al I+(1-\al)T^\Omega_{x_2},
\end{gather}
where $I$ is the identity map.
Fix $\omega\in \Omega$.
Let $x_2,x_2'\in X_2$ and $\be:=d_2(x_2,x'_2)$.
Then
\begin{eqnarray*}
\|P_{x_2}(\omega)-P_{x_2'}(\omega)\|
&\overset{\eqref{T5.3.1}}{=}&
(1-\alpha)\|T^\Omega_{x_2}(\omega)-T^\Omega_{x_2'}(\omega)\|\\
&\overset{\eqref{def-Tz}}{=}&
(1-\alpha)\|T^\Omega_1(\omega,x_2)-T^\Omega_1(\omega,x_2')\|\\
&\overset{\eqref{R2.4-3}}{\le}&(1-\alpha)d_\psi(T^\Omega(\omega,x_2),T^\Omega(\omega,x_2'))\\
&\le&(1-\alpha)d_\psi((\omega,x_2),(\omega,x_2')) \;\;(T\; \text{is nonexpansive})\\
&\overset{\eqref{T2.2-1}}{=}&
(1-\alpha)\be\\
&\le& \be.
\end{eqnarray*}	
Suppose that 
\begin{gather}\label{T5.3.3}
\|P^n_{x_2}(\omega)-P^n_{x_2'}(\omega)\|\le \be\;\;\text{for some}\;\;n\in\N.
\end{gather}	
Let 
$u_n:=P^n_{x_2}(\omega)$, $u'_n:=P^n_{x_2'}(\omega)$ and
$\gamma_n:=\|u_n-u'_n\|$.
Then
\begin{eqnarray*}
\|P^{n+1}_{x_2}(\omega)-P^{n+1}_{x_2'}(\omega)\|
&=&
\|P^{n+1}_{x_2}(u_n)-P^{n+1}_{x_2'}(u'_n)\|\notag\\
&\overset{\eqref{T5.3.1}}{=}&
\|\alpha(u_n-u'_n)+(1-\alpha)(T^\Omega_{x_2}(u_n)-T^\Omega_{x_2}(u'_n))\|\notag\\
&\le&\alpha\gamma_n +(1-\alpha)\|T^\Omega_{x_2}(u_n)-T^\Omega_{x_2}(u'_n)\| \notag\\
&\overset{\eqref{def-Tz},\eqref{T5.3.3}}{\le}&
\alpha\be+(1-\alpha)\|T^\Omega_1(u_n,x_2)-
T^\Omega_1(u'_n,x_2')\| \notag\\
&\overset{\eqref{T5.3.2}}{\le}& \alpha \be+\dfrac{1-\alpha}{\lambda}d_\psi((u_n,x_2),(u'_n,x_2'))\notag\\
&\overset{\eqref{T2.2-1}}{=}&
\al\be +\dfrac{1-\al}{\lambda}(\gamma_n+\be)\psi\left(\dfrac{\gamma_n}{\gamma_n+\be},\dfrac{\be}{\gamma_n+\be}\right) \notag\\
&\overset{\eqref{L2.3-1},\eqref{T5.3.3}}{\le}&
\al\be +\dfrac{1-\al}{\lambda}\cdot2\be\psi\left(\dfrac{1}{2},\dfrac{1}{2}\right)\\
&=&\al\be+(1-\al)\be=\be.\notag
\end{eqnarray*}	
By induction, we have
\begin{gather}\label{T5.3.4}
\gamma_n\le\be\;\;\text{for all}\;\;n\in\N.
\end{gather}
Let $n\in\N$ and define $H_n: X_2\to X_2$ by
\begin{gather}\label{H}
H_n(z):=T_2(P^n_{z}(\omega),z)\;\;\text{for all}\;\;z\in X_2.	
\end{gather}	 	
Then
\begin{eqnarray*}
d_2(H_n(x_2),H_n(x_2'))
&\overset{\eqref{H}}{=}&
d_2(T_2(u_n,x_2),T_2(u'_n,x_2'))\\
&\overset{\eqref{R2.4-3}}{\le}&
d_\psi(T(u_n,x_2),T(u'_n,x_2'))\\
&\le&
d_\psi((u_n,x_2),(u'_n,x_2'))
\;\;(T\;\text{is nonexpansive})\\
&\overset{\eqref{T2.2-1}}{=}&
(\gamma_n+\be)\psi\left(\dfrac{\gamma_n}{\gamma_n+\be},\dfrac{\be}{\gamma_n+\be}\right)\\
&\overset{\eqref{L2.3-1},\eqref{T5.3.4}}{\le}&
2\psi\left(\dfrac{1}{2},\dfrac{1}{2}\right)\be
=\lambda\be.
\end{eqnarray*}	
Thus, the map $H_n$ is Lipschitz continuous with constant $\lambda$.
By the assumption and Definition~\ref{D5.2}\eqref{D5.2-2}, there exists a sequence $\{z^k\}\subset X_2$ such that $d_2(z^k,H_n(z^k))\to 0$ as $k\to\infty$. 
Thus, there exists a $k_n\in \N$
such that 
\begin{gather}\label{T5.3.11}
b_n(\omega):=d_2(z_n,T_2(P^n_{z_n}(\omega),z_n))\le\frac{1}{n}
\end{gather}
with $z_n:=z^{k_n}$.
Observe that the sequence $\{k_n\}$ can be chosen so that $k_n\to\infty$ as $n\to\infty$.
By Lemma~\ref{L5.6},
\begin{eqnarray}\notag
a_n(\omega)
&:=&\|P^n_{z_n}(\omega)-T^\Omega_1(P^n_{z_n}(\omega),z_n)\|\\ \notag
&\overset{\eqref{def-Tz}}{=}&
\|P^n_{z_n}(\omega)-T^\Omega_{z_n}(P^n_{z_n}(\omega))\|\\
&\overset{\eqref{T5.3.1}}{=}&
\dfrac{1}{1-\al}\|P^n_{z_n}(\omega)-P^{n+1}_{z_n}(\omega)\|\to 0
\;\;\text{as}\;\;n\to\infty.\label{T5.3.12}
\end{eqnarray}	
Let $y_n:=(P^n_{z_n}(\omega),z_n)$ $(n\in\N)$.
By \eqref{R2.4-3}, \eqref{T5.3.11} and \eqref{T5.3.12},
\begin{align*}
d_\psi(y_n,T^\Omega(y_n))
&=d_\psi(y_n,(T^\Omega_1(y_n),T_2(y_n)))\\
&\le \|P^n_{z_n}(\omega)-T^\Omega_1(y_n)\|+d_2(z_n,T_2(y_n))\\
&=a_n(\omega)+b_n(\omega) \to 0\;\;\text{as}\;\;n\to\infty.
\end{align*}	
By Definition~\ref{D5.2}\eqref{D5.2-2}, $\{y_n\}$ is an approximate fixed point sequence of $T^\Omega$.
\end{proof}	

\begin{remark}	
When $\psi:=\psi_\infty$ is given by \eqref{R5.5-1}, we have $\lambda=1$, and consequently, condition~\eqref{T5.3.2} can be omitted since it is  a direct consequence of the nonexpansiveness of $T$. 
In this case, Theorem~\ref{T5.7} recaptures \cite[Theorem~4.2]{EspKir01}.
\end{remark}	

The following example illustrates Theorem~\ref{T5.7}.
\begin{example}
Let $X_1:=\R^2$ and $X_2:=[-1,1]$ be equipped with the  Euclidean norm $\|\cdot\|$ and  the standard metric $|\cdot|$,
$\Omega:=\{\omega\in\R^2\mid \|\omega\|\leq 1\}$,  $\textbf{e}_1:=(1,0)\in \Omega$, 
$X_{\Omega}:=\Omega\times X_2$ be equipped with the metric $d_{\psi_p}$  where
$\psi_p$ is given by \eqref{ppsi} for some $p\in[1,\infty)$, and $\lambda:=2\psi_p\left(\frac{1}{2},\frac{1}{2}\right)=2^{1/p}.$
Define $T^\Omega:X_{\Omega}\to X_{\Omega}$ for any $x:=(x_1,x_2)\in X_\Omega$ by 
\begin{equation}\label{E5.1}
T^\Omega(x):=(T_1^\Omega(x),T_2(x)):=\left(\frac{1}{2^{1+1/p}}\cdot x_1+\frac{1}{2^{1+1/p}} x_2\cdot\textbf{e}_1,\frac{1}{2}\langle x_1,\textbf{e}_1\rangle+\frac{1}{2}x_2\right).
\end{equation}	
For any $(x_1,x_2)\in X_\Omega$, we have
\begin{gather*}
\|T^\Omega_1(x_1,x_2)\|\le\dfrac{\|x_1\|}{2^{1+1/p}}+\dfrac{|x_2|}{2^{1+1/p}}\le\dfrac{2}{2^{1+1/p}}<1,\\
|T_2(x_1,x_2)|=\abs{\frac{1}{2}\langle x_1,\textbf{e}_1\rangle+\frac{1}{2}x_2}\le\dfrac{1}{2}+\dfrac{1}{2}=1.
\end{gather*}	
Thus, $T^{\Omega}$ is well-defined.
Let $x:=(x_1,x_2),x':=(x'_1,x'_2)\in X_\Omega$, $u_1:=x_1-x'_1$ and $u_2:=x_2-x'_2$.
Then
\begin{eqnarray*}
d_{\psi_p}(T^\Omega(x), T^\Omega(x'))
&\overset{\eqref{E5.1}}{=}&\left(\left\|\frac{1}{2^{1+1/p}}\cdot u_1+\frac{1}{2^{1+1/p}}u_2\cdot\textbf{e}_1\right\|^p+\left|\frac{1}{2}\langle u_1,\textbf{e}_1\rangle+\frac{1}{2}u_2\right|^p\right)^{\frac{1}{p}}\\
&\leq&\left[\left(\frac{1}{2}\|u_1\|+\frac{1}{2}|u_2|\right)^p+\left(\frac{1}{2}\|u_1\|+\frac{1}{2}|u_2|\right)^p\right]^{\frac{1}{p}}\\
&\overset{\eqref{E5.6-2}}{\le}&\left[2^{p}\left(\frac{1}{2^p}\|u_1\|^p+\frac{1}{2^p}|u_2|^p\right)\right]^{\frac{1}{p}}\\
&=&\left(\|u_1\|^p+|u_2|^p\right)^{\frac{1}{p}}
=d_{\psi_p}(x,x'),
\end{eqnarray*}
and consequently, $T^\Omega$ is nonexpansive.	
If $\|u_1\|\leq |u_2|$, then
\begin{eqnarray*}
\|T_1^\Omega(x)-T_1^\Omega(x')\|
&\overset{\eqref{E5.1}}{=}&\left\|\frac{1}{2^{1+1/p}}\cdot u_1+\frac{1}{2^{1+1/p}} u_2\cdot\textbf{e}_1\right\|\\
&\leq&\frac{1}{2^{1+1/p}}\|u_1\|+\frac{1}{2^{1+1/p}}|u_2|\\
&\leq& \frac{1}{2^{1/p}}|u_2|
\overset{\eqref{R2.4-3}}{\le}\frac{1}{2^{1/p}}\cdot d_{\psi_p}(x,x')
=\dfrac{1}{\lambda}\cdot d_{\psi_p}(x,x').
\end{eqnarray*}	
Thus, condition \eqref{T5.3.2}  is satisfied. 
It is clear that $(X_1,\|\cdot\|)$ is a Banach space, and $\Omega$ be a nonempty bounded closed convex subset of $X_1$. 
On the other hand, $(X_2,|\cdot|)$ has the fixed point property for continuous self-maps due to the intermediate value theorem, and consequently, $(X_2,|\cdot|)$ 
has the approximate fixed point property for Lipschitz continuous maps.
By Theorem~\ref{T5.7}, $T^\Omega$ admits an approximate fixed point sequence. 
Observe that the sequence $x^n:=(\frac{1}{n}\textbf{e}_1, -\frac{1}{n})\in X_\Omega$  $(n\in\N)$ is an approximate fixed point sequence of $T^\Omega$.
Indeed, for each $n\in \N$,  we have
$T_1^\Omega(x^n)=T_2(x^n)=0$, and consequently,
\begin{eqnarray*}
d_{\psi_p}(x^n,T^\Omega(x^n))
&=&d_{\psi_p}(x^n,(T^\Omega_1(x^n),T_2(x^n))\\
&\overset{\eqref{pmetric}}{=}&\left(\left\|\frac{1}{n}\textbf{e}_1-T_1^\Omega(x^n)\right\|^p+\left|-\dfrac{1}{n}-T_2(x^n)\right|^p\right)^{\frac{1}{p}}\\
&=&\left(\left\|\frac{1}{n}\textbf{e}_1\right\|^p+\left|\frac{1}{n}\right|^p\right)^{\frac{1}{p}}
=\frac{2^{\frac{1}{p}}}{n}\longrightarrow 0\quad\text{as } n\to\infty.
\end{eqnarray*}
\end{example}

\section{Product Constructions of Length Spaces}\label{S6}

The following definition recalls the concept of a length of a curve \cite[Definition~1.18]{BriHae99}.
\begin{definition}\label{D3.1}
Let $(X,d)$ be a metric space, and $a,b\in\R$ with $a\le b$. 
The length of a curve  is defined by
\begin{gather}\label{D3.1-1}
L(\sigma):=\sup_{a=t_0<t_1<\ldots< t_n=b}\sum_{i=1}^{n}d(\sigma(t_{i-1}),\sigma(t_i)),
\end{gather}
where the supremum is taken over the set of all partitions of $[a,b]$ with the convention that the supremum over the empty set is equal to $0$.
\end{definition}	

\begin{remark}\label{R3.2}
\begin{enumerate}
\item\label{R3.2-0}
The length of a curve is invariant under reparameterization \cite[Proposition~1.1.8]{Pap14}.
\item\label{R3.2-1}
If $L(\sigma)<\infty$, then the curve $\sigma$ is called \textit{rectifiable}.
Examples of non-rectifiable curves can be found in 
\cite[Example~1.19]{BriHae99} and
\cite[Examples~1.1.6 \& 1.1.11]{Pap14}.
\item\label{R3.2-2}
For a pair $u,v\in X$, we denote ${\rm \textbf{R}}^X_{[a,b]}(u,v)$  the set of all rectifiable curves $\sigma:[a,b]\to X$ joining $u$ and $v$, i.e., $L(\sigma)<\infty$,
$\sigma(a)=u$ and $\sigma(b)=v$.
By Definition~\ref{D3.1},
\begin{gather}\label{R3.2-3}
d(u,v)\le \inf_{\sigma\in {\rm \textbf{\rm \textbf{R}}}^{X}_{[a,b]}(u,v)}  L(\sigma):=\mathbf{d}(u,v)\;\;\text{for all}\;\;u,v\in X,
\end{gather}	
with the convention that the infimum over the empty set equals $\infty$.
It is worth to mention that $\textbf{d}$ is  a metric on $X$; see \cite[Proposition~3.2]{BriHae99} and \cite[Proposition~2.1.5]{Pap14}.
\item\label{R3.2-4}
By \cite[Remark 1.22]{BriHae99}, for any rectifiable curve $\gamma:[a,b]\to X$, there exists a rectifiable
curve $\sigma:[0,1]\to X$ such that $L(\sigma)=L(\gamma)$, and
\begin{gather}\label{R3.2-5}
L(\sigma|_{[0,t]})=t\cdot L(\sigma)\;\;\text{for all}\;\;
t\in [0,1],\\ \label{R3.2-6}
L(\sigma)=\sup_{m\in \N}\sum_{j=1}^{m}d\left(\sigma\left(\dfrac{j-1}{m}\right),\sigma\left(\dfrac{j}{m}\right)\right).
\end{gather}	
A rectifiable curve satisfying condition \eqref{R3.2-5} is called \textit{parameterized proportional to arc length} \cite[Definition~1.21]{BriHae99}.
\end{enumerate}
\end{remark}

The next definition recalls the notion of a length space; see \cite[Definition~3.1]{BriHae99} and \cite[Definition~2.1.2]{Pap14}.
\begin{definition}\label{D4.1}
A metric space $(X,d)$ is called a {length space} if the distance between any pair of points equals the greatest lower bound of the lengths of rectifiable curves joining them, i.e.,
\begin{gather*}
d(u,v)=\inf_{\sigma\in {\rm \textbf{R}}^X_{[a,b]}(u,v)} L(\sigma)\;\;\text{for all}\;\;u,v\in X.
\end{gather*}
\end{definition}

\begin{remark}
If $(X,d)$ is a length space, then ${\rm \textbf{R}}^X_{[a,b]}(u,v)$ is nonempty  for all $u,v\in X$.
\end{remark}	

The following theorem establishes a necessary and sufficient condition for a product space to be a length space.
\begin{theorem}\label{T3.5}
Let $(X_i,d_i)$ $(i=1,\ldots,n)$ be metric spaces, $M:=X_1\times\cdots\times X_n$, and $\psi\in\pmb{\Psi}_n$. 
Then $(X_i,d_i)$ $(i=1,\ldots,n)$ are length spaces if and only if $(M,d_\psi)$ is a length space.
\end{theorem}	

\begin{proof}
Suppose that $(X_i,d_i)$ $(i=1,\ldots,n)$ are length spaces.		
By Theorem~\ref{T3.10}, $(M,d_\psi)$ is a metric space.
Let $x:=(x_1,\ldots,x_n),y:=(y_1,\ldots,y_n)\in M$. 
Let $\varepsilon>0$, $\al_i:=d_i(x_i,y_i)$ $(i=1,\ldots,n)$ and $\al:=\sum_{i=1}^{n}\al_i$. 
Choose a $\be>0$ such that $\be<\frac{\varepsilon}{\al}$.
Since $\psi$ is continuous on the compact set $\Omega_n$, it is uniformly continuous on $\Omega_n$ with respect to the standard maximum norm on $\R^n$.
Thus, there exists a $\tau>0$ such that
\begin{gather}\label{T4.8-25}
|\psi(t)-\psi(t')|<\be\;\;\text{if}\;\;
|t_i-t'_i|<\tau\;\;(i=1,\ldots,n)
\end{gather}	
for all $t:=(t_1,\ldots,t_n),t':=(t'_1,\ldots,t'_n)\in\Omega_n$.
Choose a $\de>0$ such that
\begin{gather}\label{T4.8-24}
\delta<\min\left\{ \frac{\epsilon-\al\beta}{(\be+1)n},\dfrac{\al^2\tau}{\al+(n-2)\cdot\max\{\al_1,\ldots,\al_n\}}\right\}.
\end{gather} 
By the length space property and Remark~\ref{R3.2}\eqref{R3.2-4},
there exist curves $\sigma_i\in \textbf{R}^{X_i}_{[0,1]}(x_i,y_i)$ such that
\begin{gather}\label{T4.8-30}
L(\sigma_i)<\al_i+\delta, \\ \label{T4.8-21}
L(\sigma_i|_{[0,t]})=t\cdot L(\sigma_i)\;\;\text{for all}\;\;t\in[0,1]
\end{gather}
for $i=1,\ldots,n$.
Let $\sigma:=(\sigma_1,\ldots,\sigma_n)$ and $L:=\sum_{i=1}^nL(\sigma_i)$.
Then $\sigma\in \textbf{R}^{M}_{[0,1]}(x,y)$.
By \eqref{R3.2-3} and \eqref{T4.8-21},
\begin{gather}\label{T4.8-27}
\al\le L<\al +n\delta\;\;\text{and}\;\;
L(\sigma)\ge d_\psi(x,y).
\end{gather}	
Let $m\in\N$.
For each $i=1,\ldots,n$, we have 
\begin{gather}
\lambda_{ij}:=d_i\left(\sigma_i\left(\frac{j-1}{m}\right)\sigma_i\left(\frac{j}{m}\right)\right)\overset{\eqref{R3.2-3}}{\le} L\left({\sigma_i}_{|_{\left[\frac{j-1}{m},\frac{j}{m}\right]}}\right)
\overset{\eqref{T4.8-21}}{=}
\frac{L(\sigma_i)}{m} \label{T4.8-3}
\end{gather}	
for all $j=1,\ldots,m$.
By  \eqref{T4.8-30}, 
\begin{align}\notag
\sum_{j=1}^n |\al_j\cdot L(\sigma_i)-\al_i\cdot L(\sigma_j)|
&\le \sum_{j\ne i}\left( \al_j\cdot|L(\sigma_i)-\al_i|+\al_i\cdot|\al_j-L(\sigma_j)|\right)\\ 
&<\de\cdot\sum_{j\ne i}(\al_j+\al_i)=\de\left(\al+(n-2)\al_i\right) \label{T4.8-22}
\end{align}	
for all $i=1,\ldots,n$.
By \eqref{T4.8-24} and \eqref{T4.8-22},
\begin{align}\notag
\left|\frac{L(\sigma_i)}{L}-\dfrac{\al_i}{\al}\right|
&\le \dfrac{\sum_{j=1}^n|\al_j\cdot L(\sigma_i)-\al_i\cdot L(\sigma_j)|}{L\cdot\al}\\
&<\delta\cdot\dfrac{\al+(n-2)\al_i}{\al^2}<\tau
\;\;(i=1,\ldots,n).\label{T4.8-26}
\end{align}	
By \eqref{T4.8-25} and \eqref{T4.8-26},
\begin{gather}\label{T4.8-23}
\psi\left(\frac{L(\sigma_1)}{L},\ldots, \frac{L(\sigma_n)}{L}\right)<\psi\left(\frac{\al_1}{\al},\ldots, \frac{\al_n}{\al}\right)+\be.
\end{gather}	
For $j=1,\ldots,m$, 
\begin{eqnarray}\notag
d_\psi\left(\sigma\left(\frac{j-1}{m}\right),\sigma\left(\frac{j}{m}\right)\right)
&\overset{\eqref{T2.2-1}}{=}&\left(\sum_{i=1}^n\lambda_{ij}\right)\cdot
\psi\left(\frac{\lambda_{1j}}{\sum_{i=1}^n\lambda_{ij}},\ldots, \frac{\lambda_{nj}}{\sum_{i=1}^n\lambda_{ij}}\right)\\ 
&\overset{\eqref{L2.3-1},\eqref{T4.8-3} }{\le}&\dfrac{L}{m}\cdot
\psi\left(\frac{L(\sigma_1)}{L},\ldots, \frac{L(\sigma_n)}{L}\right). \label{T4.8-8}
\end{eqnarray}	
Then
\begin{eqnarray*}
L(\sigma)
&\overset{\eqref{R3.2-6}}{=}&\sup_{m\in\N}\sum_{j=1}^{m}d_\psi\left(\sigma\left(\frac{j-1}{m}\right),\sigma\left(\frac{j}{m}\right)\right)\\
&\overset{\eqref{T2.2-1},\eqref{T4.8-8}}{\le}& L\cdot
	\psi\left(\frac{L(\sigma_1)}{L},\ldots, \frac{L(\sigma_n)}{L}\right)\\
&\overset{\eqref{T4.8-27}, \eqref{T4.8-23}}{<}&
(\al+n\delta)\cdot\left[
	\psi\left(\frac{\al_1}{\al},\ldots, \frac{\al_n}{\al}\right)+\be\right]\\
&=&\al\cdot\psi\left(\frac{\al_1}{\al},\ldots, \frac{\al_n}{\al}\right)+\beta(\al+n\delta)+n\delta\cdot\psi\left(\frac{\al_1}{\al},\ldots, \frac{\al_n}{\al}\right)\\
	&\overset{\eqref{R2.1-4},\eqref{T2.2-1}}{\le}& d_\psi(x,y)+\beta(\al+n\delta)+n\delta\\
&\overset{\eqref{T4.8-24}}{<}&d_\psi(x,y)+\varepsilon.
\end{eqnarray*}
Thus, 
\begin{gather*}
d_\psi(x, y)=\inf_{\sigma\in {\rm \textbf{R}}^{M}_{[0,1]}(x,y)}L(\sigma).
\end{gather*} 
By Definition~\ref{D4.1}, $(M,d_\psi)$ is a length space. 

Suppose that $(M,d_\psi)$ is a length space.
Let $i=1,\ldots,n$ and $x_{i}, y_{i}\in X_i$.
Define $x,y\in M$ by \eqref{T4.8-9} for some $x_j\in X_j$ ($j\ne i$).
Let $\varepsilon>0$.
By \eqref{T2.2-1} and \eqref{T4.8-9}, 
\begin{gather}\label{T4.8-4}
	d_\psi(x,y)=d_i(x_i,y_i)\cdot\psi\left(\mathbf{e}_i\right)=d_i(x_i,y_i).
\end{gather}
By Definition~\ref{D4.1}, there exists a curve $\sigma:=(\sigma_{1},\ldots,\sigma_{n})\in {\rm \textbf{R}}^{M}_{[a,b]}(x, y)$ such that
\begin{gather}\label{T4.8-5}
	d_\psi(x,y)+\epsilon>L(\sigma).
\end{gather}	
Then
\begin{eqnarray*}
d_i(x_i,y_i)+\epsilon
&\overset{\eqref{T4.8-4}}{=}&d_\psi(x,y)+\epsilon\\
&\overset{\eqref{T4.8-5}}{>}&	L(\sigma)\\
&\overset{\eqref{D3.1-1}}{=}&\sup_{a=t_0<t_1<\ldots< t_n=b}\sum_{i=1}^{n}d_\psi(\sigma(t_{i-1}),\sigma(t_i))\\ 
&\overset{\eqref{R2.4-3}}{\ge}& \sup_{a=t_0<t_1<\ldots< t_n=b}\sum_{i=1}^{n}d_i(\sigma_i(t_{i-1}),\sigma_i(t_i))
=L(\sigma_i).
\end{eqnarray*}	
Observe that $\sigma_i\in {\rm \textbf{R}}^{X_i}_{[a,b]}(x_i,y_i)$. 
Thus, 
\begin{gather*}
d_i(x_{i}, y_{i})=\inf_{\gamma \in {\rm \textbf{R}}^{X_i}_{[a,b]}(x_i,y_i)}L(\gamma).
\end{gather*} 
By Definition~\ref{D4.1}, $(X_i,d_i)$ is a length space. 
\end{proof}	

\begin{remark}
When $\psi:=\psi_p$ is defined by 
\eqref{ppsi} with  $p\in[1,\infty]$,  Theorem~\ref{T3.5} recaptures \cite[Proposition~5.3(1) \& Exercises~5.5(1)]{BriHae99}.
\end{remark}	

The following statement recalls a characterization of a length space; cf. \cite[p.32]{BriHae99} and \cite[Theorem~3.4]{Iof17}.

\begin{lemma}\label{L3.6}
A complete metric space $(X,d)$ is a length space if and only if for all $x, y\in X$ and $\epsilon>0$, there exists a  $\bx\in X$ such that
\begin{gather}\label{appmidpoint}
\max\{d(x,\bx), d(y,\bx)\}\leq\frac{1}{2}d(x,y)+\epsilon.
\end{gather}
\end{lemma}	

The proof of Theorem~\ref{T3.5} uses the  definition of a length space (Definition~\ref{D4.1}).
When $(X_i,d_i)$ $(i=1,\ldots,n)$ are complete metric spaces, an alternative proof can be provided by employing Lemma~\ref{L3.6}.

\begin{proof}[Proof of Theorem~\ref{T3.5} under the completeness assumption]
By Remark~\ref{R2.4}\eqref{R2.4-2} and Theorem~\ref{T3.10},  $(M,d_\psi)$ is a complete metric space.
\sloppy
	
Suppose that $(X_i,d_i)$ $(i=1,\ldots,n)$ are length spaces.		
Let $x:=(x_1,\ldots,x_n), y:=(y_1,\ldots,y_n)\in M$,
$\varepsilon>0$, $\al_i:=d_i(x_i,y_i)$ $(i=1,\ldots,n)$ and $\al:=\sum_{i=1}^{n}\al_i$. 
Choose a $\be>0$ such that $\be<\frac{2\varepsilon}{\al}$.
Then there exists a $\tau>0$ such that condition \eqref{T4.8-25} is satisfied for all $t:=(t_1,\ldots,t_n)$, $t':=(t'_1,\ldots,t'_n)\in\Omega_n$.
Choose a  $\de>0$ such that
\begin{gather}\label{T4.8-16}
\delta<\min\left\{\dfrac{\al}{4n},\dfrac{\al^2\tau}{4\left(\al+(n-2)\cdot\max\{\al_1,\ldots,\al_n\}\right)},\dfrac{2\epsilon-\alpha\beta}{2(\beta+1)n}\right\}.
\end{gather} 
By Lemma~\ref{L3.6}, there exist points $\bx_i\in X_i$ 
$(i=1,\ldots,n)$ such that
\begin{gather}\label{T4.8-12}
\max\{d_i(x_i,\bx_i),d_i(y_i,\bx_i)\}\le \dfrac{\al_i}{2}+\delta\;\;
(i=1,\ldots,n).	
\end{gather}	
Let $\bx:=(\bx_1,\ldots,\bx_n)$, $\theta_i:=d_i(x_i,\bx_i)$ $(i=1,\ldots,n)$, and $\theta:=\sum_{i=1}^{n}\theta_i$.
We are going to show that
\begin{gather}\label{T4.8-14}
	d_\psi(x,\bx)<\dfrac{1}{2}d_\psi(x,y)+\varepsilon.
\end{gather}	
By \eqref{T4.8-12} and the triangle inequality,
\begin{gather}\label{T4.8-40}
\frac{\al_i}{2}+\delta\ge \theta_i\ge \al_i-d_i(y_i,\bx_i)\ge\frac{\al_i}{2}-\delta\;\;(i=1,\ldots,n).
\end{gather}	
By \eqref{T4.8-16} and \eqref{T4.8-40},
\begin{gather}\label{T4.8-13}
\left|\theta_i-\dfrac{\al_i}{2}\right|\le \delta\;\;(i=1,\ldots,n)\;\;\text{and}\;\;
\dfrac{\al}{2}+n\delta\ge \theta \ge \dfrac{\al}{2}-n\delta>\dfrac{\al}{4}.
\end{gather}	
By \eqref{T4.8-13}, 
\begin{align}\notag
\sum_{j=1}^n |\al_j\cdot \theta_i-\al_i\cdot \theta_j|
&\le \sum_{j\ne i}\left( \al_j\cdot\left|\theta_i-\frac{\al_i}{2}\right|+\al_i\cdot\left|\frac{\al_j}{2}-\theta_j\right|\right)\\ \label{T4.8-17}
&\le\de\cdot\sum_{j\ne i}(\al_j+\al_i)=\de\left(\al+(n-2)\al_i\right)\;\;(i=1,\ldots,n).
\end{align}	
By \eqref{T4.8-16}, \eqref{T4.8-13} and \eqref{T4.8-17},
\begin{align}\label{T4.8-41}
\left|\frac{\theta_i}{\theta}-\dfrac{\al_i}{\al}\right|
\le \dfrac{\sum_{j=1}^n|\al_j\cdot \theta_i-\al_i\cdot\theta_j|}{\theta\cdot\al}
<\delta\cdot\dfrac{4\left(\al+(n-2)\al_i\right)}{\al^2}<\tau\;\;(i=1,\ldots,n).
\end{align}	
By \eqref{T4.8-25} and \eqref{T4.8-41},
\begin{gather}\label{T4.8-19}
\psi\left(\dfrac{\theta_1}{\theta},\ldots,\dfrac{\theta_n}{\theta}\right)<\psi\left(\dfrac{\al_1}{\al},\ldots,\dfrac{\al_n}{\al}\right)+\be.
\end{gather}	
Then
\begin{eqnarray*}
d_\psi(x,\bx)
&\overset{\eqref{T2.2-1}}{=}&\theta\cdot\psi\left(\dfrac{\theta_1}{\theta},\ldots,\dfrac{\theta_n}{\theta}\right)\\
&\overset{\eqref{T4.8-13},\eqref{T4.8-19}}{<}& \left(\dfrac{\al}{2}+n\delta\right)\cdot\left[\psi\left(\dfrac{\al_1}{\al},\ldots,\dfrac{\al_n}{\al}\right)+\be\right]\\
&=& \dfrac{1}{2}\al\cdot\psi\left(\dfrac{\al_1}{\al},\ldots,\dfrac{\al_n}{\al}\right)+\be\left(\dfrac{\al}{2}+n\delta\right)+n\delta\cdot\psi\left(\frac{\al_1}{\al},\ldots, \frac{\al_n}{\al}\right)\\
&\overset{\eqref{R2.1-4},\eqref{T2.2-1}}{\le}& \dfrac{1}{2}d_\psi(x,y)+\be\left(\dfrac{\al}{2}+n\delta\right)+n\delta\\&\overset{\eqref{T4.8-16}}{<}& \dfrac{1}{2}d_\psi(x,y)+\varepsilon.
\end{eqnarray*}	
Thus, condition \eqref{T4.8-14} holds true.
Using a similar argument, we obtain 
$d_\psi(y,\bx)<\frac{1}{2}d_\psi(x,y)+\varepsilon.$
Thus, 
\begin{align}\label{T4.8-10}
\max\{d_\psi(x,\bx), d_\psi(y,\bx)\}\leq\frac{1}{2}d_\psi(x,y)+\epsilon.
\end{align}
By Lemma~\ref{L3.6}, $(M,d_\psi)$ is a length space.

Suppose that $(M,d_\psi)$ is a length space.
Let $i=1,\ldots,n$ and  $x_{i}, y_{i}\in X_i$.
Define $x,y\in M$ by \eqref{T4.8-9} for some
$x_j\in X_j$ ($j\ne i$).
Then condition \eqref{T4.8-4} holds true.
Let $\varepsilon>0$.
By Lemma~\ref{L3.6}, there exists a $\bx:=(\bx_1,\ldots,\bx_n)\in M$ such that condition \eqref{T4.8-10} is satisfied.
By \eqref{R2.4-3}, 
\begin{gather}\label{T4.8-11}
d_i(x_i,\bx_i)\le d_{\psi}(x,\bx),\;\;d_i(y_i,\bx_i)\le d_{\psi}(y,\bx).
\end{gather}	
By \eqref{T4.8-4}, \eqref{T4.8-10} and \eqref{T4.8-11},
\begin{align*}
\max\{d_i(x_i,\bx_i), d_i(y_i,\bx_i)\}\leq\frac{1}{2}d_i(x_i,y_i)+\epsilon.
\end{align*}
By Lemma~\ref{L3.6}, $(X_i,d_i)$ is a length space.
\end{proof}

\begin{remark}
The completeness assumption is not essential for Theorem~\ref{T3.5}. Nevertheless, we include the above proof to illustrate that the class of continuous convex functions provides a convenient setting for handling the technicalities involved in the product construction of length spaces.
\end{remark}

\section{Product Constructions of Geodesic Spaces}\label{S7}
The following definition recalls the notion of a geodesic space; cf. \cite[Definition~1.3]{BriHae99} and \cite[Definition~2.4.1]{Pap14}.
\begin{definition} \label{D4.2}
A metric space $(X,d)$ is called a {geodesic space} if for any  $x,y\in X$, there exists a geodesic joining $x$ and $y$, i.e., there exists a curve $\gamma:[0,d(x,y)]\to X$ such that
$\gamma(0)=x$, $\gamma(d(x,y))=y$, and
\begin{gather}\label{D4.2-1}
d(\gamma(t),\gamma(t'))=|t-t'|\;\;\text{for all}\;\;t,t'\in[0,d(x,y)].	
\end{gather}
\end{definition}

\begin{remark}
Examples of geodesic spaces can be found in \cite[Example~2.4.3]{Pap14}.
Every geodesic space is a length space \cite[Proposition~2.4.2]{Pap14}, and the converse implication does not hold in general \cite[Example~2.4.4]{Pap14}.
By the Hopf–Rinow theorem \cite[Proposition~3.7]{BriHae99}, a length space is a geodesic space if it is complete and locally compact.
\end{remark}	

\begin{proposition}\label{P4.5}
A metric space $(X,d)$ is a geodesic space if and only if for any $x,y\in X$, there exists a curve $\sigma:[0,1]\to X$ such that
$\sigma(0)=x$, $\sigma(1)=y$, and
\begin{gather}\label{D4.2-10}
d(\sigma(t),\sigma(t'))=d(x,y)\cdot|t-t'|\;\;\text{for all}\;\;t,t'\in[0,1].	
\end{gather}
\end{proposition}	

\begin{proof}
Suppose that $(X,d)$ is a geodesic space.
Let $x,y\in X$.
By Definition~\ref{D4.2}, there exists a curve $\gamma:[0,d(x,y)]\to X$ such that $\gamma(0)=x$, $\gamma(d(x,y))=y$, and condition \eqref{D4.2-1} is satisfied.
Let $\sigma(t):=\gamma(t\cdot d(x,y))$ for all $t\in[0,1]$.
Then $\sigma(0)=x$, $\sigma(1)=y$, and condition \eqref{D4.2-10} holds true.

We now prove the converse implication.
Let $x,y\in X$.
If $x=y$, then by Definition~\ref{D4.1-1}\eqref{D4.1-1.1}, the curve $\gamma:[0,0]\to X$ given by $\gamma(0)=x$ is a geodesic.
Suppose that $x\ne y$.
By the assumption, there exists a curve $\sigma:[0,1]\to X$ such that
$\sigma(0)=x$, $\sigma(1)=y$, and condition \eqref{D4.2-10} is satisfied.
By Definition~\ref{D4.1-1}\eqref{D4.1-1.2}, $\sigma$ is a constant speed geodesic with number $d(x,y)>0$.
By Proposition~\ref{P4.2},  the curve $\gamma:[0,d(x,y)]\to X$ defined by
\begin{gather*}
\gamma(t):=\sigma\left(\frac{t}{d(x,y)}\right)
\;\;\text{for all}\;\;t\in[0,d(x,y)]
\end{gather*}
is a geodesic.
Note that  $\gamma(0)=x$ and $\gamma(d(x,y))=y$.
By Definition~\ref{D4.2}, $(X,d)$ is a geodesic space.
\end{proof}	

\begin{remark}
The choice of the interval $[0,1]$ in Proposition~\ref{P4.5} is merely a normalization.
More precisely, one can consider  an arbitrary interval $[a,b]$ with $a<b$ and define a curve $\tilde\sigma:[a,b]\to X$ satisfying
$\tilde\sigma(a)=x$, $\tilde\sigma(b)=y$, and
\begin{gather*}
d(\tilde\sigma(t),\tilde\sigma(t'))=\frac{d(x,y)}{b-a}\cdot|t-t'|
\;\;\text{for all}\;\; t,t'\in[a,b].	
\end{gather*}	
Let $\sigma:[0,1]\to X$ be given by $\sigma(t):=\tilde\sigma(a+t(b-a))$ for all $t\in[0,1]$.
Then $\sigma(0)=x$, $\sigma(1)=y$, and condition \eqref{D4.2-10} is satisfied.
\end{remark}	

The following theorem establishes a necessary and sufficient condition for the product space to be a geodesic space.
\begin{theorem}\label{T4.9}
Let $(X_i,d_i)$ $(i=1,\ldots,n)$ be  metric spaces, $M:=X_1\times\cdots\times X_n$, and $\psi\in\pmb{\Psi}_n$. 
Then  $(X_i,d_i)$ $(i=1,\ldots,n)$ are geodesic spaces if and only if $(M,d_\psi)$ is a geodesic space.
\end{theorem}

\begin{proof}
By Theorem~\ref{T3.10}, $(M,d_\psi)$ is a metric space.
 Suppose that $(X_i,d_i)$ $(i=1,\ldots,n)$ are geodesic spaces.
 Let
 $x:=(x_1,\ldots,x_n),y:=(y_1,\ldots,y_n)\in M$.
 By Proposition~\ref{P4.5}, there exist curves $\sigma_i:[0,1]\to X_i$ satisfying $\sigma_i(0)=x_i$, $\sigma_i(1)=y_i$, and
 \begin{gather}\label{T4.9-5}
 d_i(\sigma_i(t),\sigma_i(t'))=d_i(x_i,y_i)\cdot|t-t'|\;\;\text{for all}\;\;t\in[0,1]
 \end{gather}	
 $(i=1,\ldots,n)$.
 Let $\sigma:=(\sigma_1,\ldots,\sigma_n)$.
 Then $\sigma(0)=x$ and $\sigma(1)=y$.
 By \eqref{T2.2-1} and \eqref{T4.9-5},
 \begin{align*}
 d_\psi(\sigma(t),\sigma(t'))
 &=d_\psi((\sigma_1(t),\ldots,\sigma_n(t)),(\sigma_1(t'),\ldots,\sigma_n(t')))\\
 &=\left(\sum_{i=1}^n d_i(x_i,y_i)\right)\cdot\psi\left(\dfrac{d_1(x_1,y_1)}{\sum_{i=1}^nd_i(x_i,y_i)},\ldots,\dfrac{d_n(x_n,y_n)}{\sum_{i=1}^nd_i(x_i,y_i)}\right)\cdot|t-t'| \\
 &=d_\psi(x,y)\cdot|t-t'|
 \end{align*}	
 for all $t,t'\in [0,1]$.
 By Proposition~\ref{P4.5},
 $(M,d_\psi)$ is a geodesic space.

Suppose that $(M,d_\psi)$ is a geodesic space.
 Let $i=1,\ldots,n$ and  $x_{i}, y_{i}\in X_i$.
 Define $x,y\in M$ by \eqref{T4.8-9} for some
 $x_j\in X_j$ ($j\ne i$).
 By Proposition~\ref{P4.5}, there exists a curve $\sigma:[0,1]\to M$ satisfying $\sigma(0)=x$, $\sigma(1)=y$, and
 \begin{gather}\label{T4.9-6}
 d_\psi(\sigma(t),\sigma(t'))=d_\psi(x,y)\cdot|t-t'|\;\;\text{for all}\;\;t,t'\in[0,1].
 \end{gather}	
Note that $\sigma:=(\sigma_1,\ldots,\sigma_n)$ for some $\sigma_i: [0,1]\to X_i$ $(i=1,\ldots,n)$. 
Fix $i=1,\ldots,n$.
Then $\sigma_i(0)=x_i$, $\sigma_i(1)=y_i$, and
 \begin{gather}
 d_i(\sigma_i(t),\sigma_i(t'))\overset{\eqref{R2.4-3}}{\le} d_\psi(\sigma(t),\sigma(t'))\overset{\eqref{T4.9-6}}{=}d_\psi(x,y)\cdot|t-t'|\overset{\eqref{T4.8-4}}{=}d_i(x_i,y_i)\cdot|t-t'|\label{T4.9-10}
 \end{gather}
 for all $t,t'\in [0,1]$. 
 We have
 \begin{gather}\label{T4.9-12}
d_i(x_i, \sigma_i(t))=t\cdot d_i(x_i,y_i)\;\;\text{and}\;\; d_i(\sigma_i(t), y_i)= (1-t)\cdot d_i(x_i,y_i)
 \end{gather}	
for all $t\in[0,1]$.
 Indeed, for any $t\in[0,1]$, it follows from \eqref{T4.9-10} that
 $d_i(x_i, \sigma_i(t))\le t\cdot d_i(x_i,y_i)$ and $d_i(\sigma_i(t), y_i)\le (1-t)\cdot d_i(x_i,y_i)$, and consequently,
 \begin{gather*}
 d_i(x_i,y_i)\leq d_i(x_i, \sigma_i(t))+d_i(\sigma_i(t), y_i)\leq t\cdot d_i(x_i,y_i)+(1-t)\cdot d_i(x_i,y_i)=d_i(x_i,y_i).
 \end{gather*}
Hence, condition \eqref{T4.9-12} is satisfied.
Then
\begin{gather}
d_i(\sigma_i(t),\sigma_i(t'))\geq \left|d_i(x_i,\sigma_i(t))-d_i(x_i,\sigma_i(t'))\right|\overset{\eqref{T4.9-12}}{=} d_i(x_i,y_i)\cdot|t-t'|\label{T4.9-13}
\end{gather}
for all $t,t'\in[0,1]$.
By \eqref{T4.9-10} and \eqref{T4.9-13},
\begin{gather*}
d_i(\sigma_i(t),\sigma_i(t'))=d_i(x_i,y_i)\cdot|t-t'|\;\;\text{for all}\;\;t,t'\in[0,1].
\end{gather*}	
By Proposition~\ref{P4.5}, $(X_i,d_i)$ is a geodesic space.
\end{proof}	

\begin{remark}
When $\psi:=\psi_p$ is defined by 
\eqref{ppsi} with  $p\in[1,\infty]$,  Theorem~\ref{T4.9} recaptures \cite[Proposition~5.3(2) \& Exercises~5.5(1)]{BriHae99} and improves \cite[Proposition~2.6.6]{Pap14}.	
\end{remark}	

The next statement recalls the midpoint characterization of a geodesic space; cf. \cite[Remarks~1.4(1)]{BriHae99} and \cite[Theorem~2.6.2(ii) \& (iv)]{Pap14}.
\begin{lemma}\label{L4.4}
A complete metric space $(X,d)$ is a geodesic space if and only if for any $x, y\in X$, there exists a $\bx\in X$ such that
\begin{gather}\label{midpoint}
d(x,\bx)=d(y,\bx)=\frac{1}{2}d(x,y).
\end{gather}
\end{lemma}

The proof of Theorem~\ref{T4.9} uses the definition of a geodesic space. 
When $(X_i,d_i)$ $(i=1,\ldots,n)$ are complete metric spaces, an alternative proof can be provided by employing Lemma~\ref{L4.4}.

\begin{proof}[Proof of Theorem~\ref{T4.9} under the completeness assumption]
By Remark~\ref{R2.4}\eqref{R2.4-2} and Theorem~\ref{T3.10},  $(M,d_\psi)$ is a complete metric space.
Suppose that $(X_i,d_i)$ $(i=1,\ldots,n)$ are geodesic spaces.		
Let $x:=(x_1,\ldots,x_n), y:=(y_1,\ldots,y_n)\in M$. 
By Lemma~\ref{L4.4}, there exist points $\bx_i\in X_i$ such that
\begin{gather}\label{T4.9-4}
d_i(x_i,\bx_i)=d_i(y_i,\bx_i)=\frac{1}{2}d_i(x_i,y_i)\;\;(i=1,\ldots,n).
\end{gather}	
Let $\theta_i:=d_i(x_i,\bx_i)$,
$\al_i:=d_i(x_i,y_i)$ $(i=1,\ldots,n)$, $\theta:=\sum_{i=1}^{n}\theta_i$, $\al:=\sum_{i=1}^{n}\al_i$, and $\bx:=(\bx_1,\ldots,\bx_n)$.
By \eqref{T2.2-1} and \eqref{T4.9-4},
\begin{align*}
d_\psi(x,\bx)
=\theta\cdot\psi\left(\dfrac{\theta_1}{\theta},\ldots,\dfrac{\theta_n}{\theta}\right)
=\dfrac{1}{2} \al\cdot \psi\left(\dfrac{\al_1}{\al},\ldots,\dfrac{\al_n}{\al}\right)
=\dfrac{1}{2}d_\psi(x,y).
\end{align*}	
Using the similar argument, we obtain $d_\psi(y,\bx)=\frac{1}{2}d_\psi(x,y).$
By Lemma~\ref{L4.4}, $(M,d_\psi)$ is a geodesic space.
	
Suppose that $(M,d_\psi)$ is a geodesic space.
Fix $i=1,\ldots,n$ and $x_{i}, y_{i}\in X_i$.
Let $x,y\in M$ be given by \eqref{T4.8-9} for some
$x_j\in X_j$ ($j\ne i$).
Then condition \eqref{T4.8-4} is satisfied.	
By Lemma~\ref{L4.4}, there exists a $\bx:=(\bx_1,\ldots,\bx_n)\in M$ such that
\begin{gather}\label{T4.9-1}
d_\psi(x,\bx)=d_\psi(y,\bx)=\dfrac{1}{2}d_\psi(x,y)=\dfrac{1}{2}d_i(x_i,y_i).	
\end{gather}	
Note that condition \eqref{T4.8-11} is satisfied.
From the latter and \eqref{T4.9-1},
\begin{gather}\label{T4.9-2}
d_i(x_i,\bx_i)\le \dfrac{1}{2}d_i(x_i,y_i)\;\;\text{and}\;\;d_i(y_i,\bx_i)\le \dfrac{1}{2}d_i(x_i,y_i).
\end{gather}	
	By the triangle inequality, 
	\begin{gather}\label{T4.9-3}
		d_i(x_i,y_i)\leq d_i(x_i,\bx_i)+d_i(y_i,\bx_i).
	\end{gather}	
	By \eqref{T4.9-2} and \eqref{T4.9-3}, we have $d_i(x_i,\bx_i)=d_i(y_i,\bx_i)=\frac{1}{2}d_i(x_i,y_i).$
	By Lemma~\ref{L4.4}, $(X_i,d_i)$ is a geodesic space. 
\end{proof}	

We are going to establish a formula to compute length of a curve in a product space via lengths of its component curves.  
First, we recall the notion of an \textit{affinely (reparametrized) geodesic} \cite[Definition 2.2.8]{Pap14}.
\begin{definition}
Let $(X,d)$ be a metric space. 
A curve $\sigma:[a,b]\to X$ is called an \textit{affinely  geodesic} if either $\sigma$ is constant  or 
there exists a geodesic $\sigma': [c,d]\to X$ such that $\sigma=\sigma'\circ\theta$, where $\theta: [a,b]\to[c,d]$ is uniquely defined by
\begin{gather*}
\theta(t):=\frac{(d-c)t+bc-ad}{b-a}\;\;\text{for all}\;\;t\in[a,b].
\end{gather*}
\end{definition}

The next statement provides characterizations of an affinely geodesic; cf \cite[Propositions 2.2.9 and 2.2.10]{Pap14}.
\begin{proposition}\label{L4.11}
Let $(X,d)$ be a metric space, and $\sigma:[0,1]\to X$.
\begin{enumerate}
\item\label{L4.11-1}
If $\sigma$ is an affinely geodesic, then it is a constant speed geodesic with number $\lambda:=L(\sigma)$.
\item\label{L4.11-2}
If $\sigma$ is a constant speed geodesic with number $\lambda\ge0$, then it is an affinely geodesic and $\lambda=L(\sigma)$.
\end{enumerate}
\end{proposition}

The following statement gives a formula for computing length of curves in product spaces.
\begin{theorem}\label{T4.12}
Let $(X_i,d_i)$ be geodesic spaces, $x_i,y_i\in X_i$,
$\sigma_i:[0,1]\to X_i$ be an affinely geodesic with
$\sigma_i(0)=x_i$ and $\sigma_i(1)=y_i$ $(i=1,\ldots,n)$,
$x:=(x_1,\ldots,x_n)$, $y:=(y_1,\ldots,y_n)$,
$M:=X_1\times\cdots\times X_n$, $\psi\in\pmb{\Psi}_n$, and $\sigma:=(\sigma_1,\ldots,\sigma_n)$.
Then $\sigma$ is an affinely geodesic joining $x$ and $y$ in the geodesic space $(M,d_\psi)$, and
\begin{gather}\label{T4.12-1}
L(\sigma)=\left(\sum_{i=1}^{n}L(\sigma_i) \right)\cdot\psi\left(\frac{L(\sigma_1)}{\sum_{i=1}^{n}L(\sigma_i)},\ldots, \frac{L(\sigma_n)}{\sum_{i=1}^{n}L(\sigma_i)}\right).
\end{gather}
\end{theorem}

\begin{proof}
By	Theorem~\ref{T4.9}, $(M,d_\psi)$ is a geodesic space.
It is clear that $\sigma$ joins $x$ and $y$.
Let $t,t'\in[0,1]$.
By Definition~\ref{D4.1-1}\eqref{D4.1-1.2} and Proposition~\ref{L4.11}\eqref{L4.11-1},
\begin{gather}\label{T4.12-2}
d_i(\sigma_i(t),\sigma_i(t'))=L(\sigma_{i})\cdot|t-t'|\;\;(i=1,\ldots,n).
\end{gather}	
Let $\lambda$ be the number in the right-hand side of \eqref{T4.12-1}.
By \eqref{T2.2-1} and \eqref{T4.12-2},
\begin{align*}
d_\psi(\sigma(t),\sigma(t'))
=d_\psi((\sigma_1(t),\ldots,\sigma_n(t)),(\sigma_1(t'),\ldots,\sigma_n(t')))
=\lambda\cdot|t-t'|.
\end{align*}
By Definition~\ref{D4.1-1}\eqref{D4.1-1.2}, the curve $\sigma$ is a constant speed geodesic with number $\lambda$. 
In view of Proposition~\ref{L4.11}\eqref{L4.11-2}, the curve $\sigma$ is  an affinely geodesic  and $\lambda=L(\sigma)$.
This completes the proof.
\end{proof}

\begin{remark}
When $\psi:=\psi_p$	is given by \eqref{ppsi} with $p\in[1,\infty]$, Theorem~\ref{T4.12} recaptures
 \cite[Proposition~2.6.7]{Pap14}.
\end{remark}	

The following example illustrates Theorem~\ref{T4.12}.
\begin{example}
Let $X_i:=\R$ be equipped with the metric
$d_i(x,y):=\left|x^{2i-1}-y^{2i-1}\right|$ for all $x,y\in \R$ $(i=1,\ldots,n)$.
For each $i=1,\ldots,n$, let $u,v\in X_i$ and
define a curve $\xi_i:[0,1]\to\R$ by
\begin{gather*}
\xi_i(t):=\left((1-t)\cdot u^{2i-1}+t\cdot  v^{2i-1}\right)^{\frac{1}{2i-1}}
\;\;\text{for all}\;\;t\in[0,1].
\end{gather*}	
Then $d_i(\xi_i(t),\xi_i(t'))=d_i(u,v)\cdot |t-t'|$
for all $t,t'\in[0,1]$.
By Proposition~\ref{P4.5}, $(X_i,d_i)$ $(i=1,\ldots,n)$ are geodesic spaces.
Let  $M:=X_1\times\cdots\times X_n$, and $\psi\in\pmb{\Psi}_n$.
By Theorem~\ref{T4.9}, $(M,d_\psi)$ is a geodesic space.

For each $i=1,\ldots,n$, consider a curve $\sigma_i:[0,1]\to \R$ given by $\sigma_{i}(t):=(i\cdot t)^{\frac{1}{2i-1}}$ for all $t\in[0,1]$. 
It is clear that $\sigma_i$ joins $0$ and $i^{\frac{1}{2i-1}}$ $(i=1,\ldots,n)$.
Let $\sigma:=(\sigma_1,\ldots,\sigma_n)$.
Then $\sigma$ joins $(0,\ldots,0)$ and $(1,\ldots,i^{\frac{1}{2i-1}},\ldots,n^{\frac{1}{2n-1}})$.
We have $d_i(\sigma_i(t),\sigma_i(t'))=i\cdot|t-t'|$ for all $t,t'\in[0,1]$.
By Definition~\ref{D4.1-1}\eqref{D4.1-1.2},  $\sigma_i$ is a constant speed geodesic with number $i$.
By  Proposition~\ref{L4.11}\eqref{L4.11-2}, $\sigma_i$ is an affinely geodesic and $L(\sigma_i)=i$. 
By Theorem~\ref{T4.12}, $\sigma$ is an affinely geodesic, and\sloppy
\begin{align*}
L(\sigma)
&=\left(\sum_{i=1}^{n}L(\sigma_i) \right)\cdot\psi\left(\frac{L(\sigma_1)}{\sum_{i=1}^{n}L(\sigma_i)},\ldots, \frac{L(\sigma_n)}{\sum_{i=1}^{n}L(\sigma_i)}\right)\\
&=(1+\cdots+n)\cdot\psi\left(\frac{1}{1+\cdots+n},\ldots,\frac{n}{1+\cdots+n}\right)\\
&=\frac{n(n+1)}{2}\cdot\psi\left(\frac{2}{n(n+1)},\frac{4}{n(n+1)},\ldots,\frac{2n}{n(n+1)}\right).
\end{align*}
If $\psi:=\psi_p$ is given by \eqref{ppsi} with $p\in[1,\infty]$, then
\begin{gather*}
L(\sigma)=
\begin{cases}	
(1^p+2^p+\cdots+n^p)^{\frac{1}{p}}& \text{\rm if } p\in[1,\infty),\\
n & \text{\rm if } p=\infty.
\end{cases}
\end{gather*}
\end{example}

\section*{Author contributions}
The authors contributed equally to this work.

\section*{Funding}
No funding was received for conducting this study.

\section*{Data availability statement}
No datasets were generated or analyzed during the current study.

\section*{Declarations}

\subsection*{Conflict of interest}
The authors declare no conflict of interest.

%\section*{Funding}
%The research has been supported by the Postdoctoral Scholarship Programme of the Vingroup Innovation Foundation (VinIF) code VINIF.2023.STS.54.
%The research is supported by Vietnam National Program for the Development of Mathematics 2021-2030 under grant number B2023-CTT-09.

%\section*{Acknowledgement}The author wishes to thank Professor Alexander Kruger for comments and suggestions.

%\addcontentsline{toc}{section}{References}
%\bibliography{BUCH-kr,Kruger,KR-tmp}

\begin{thebibliography}{99}
\bibitem{Pon86}
Poincar\'e~H. Sur les courbes d\'efinies par les \'equations diff\'erentielles (quatri\`eme partie). Journal de math\'ematiques pures et appliqu\'ees.
1886;\hspace{0pt}2:151--217.
	
\bibitem{Ban22}
Banach~S. Sur les op{\'e}rations dans les ensembles abstraits et leur
application aux {\'e}quations int{\'e}grales. Fundamenta Mathematicae.
1922;\hspace{0pt}3:133--181.
	
\bibitem{Bro12}
Brouwer~LEJ. {\"U}ber abbildungen von mannigfaltigkeiten. Mathematische Annalen. 1912;\hspace{0pt}71:97--115.
	
\bibitem{Sch30}
Schauder~J. Der fixpunktsatz in funktionalr\"aumen. Studia Mathematica. 1930;\hspace{0pt}2(1):171--180.
	
\bibitem{AnsSah23}
Ansari~QH, Sahu~DR. Fixed point theory and variational principles in metric spaces. Cambridge: Cambridge University Press; 2023.
	
\bibitem{GolAgaKum21}
Gopal~D, Agarwal~P, Kumam~P, editors. Metric structures and fixed point theory. Boca Raton: Chapman \& Hall/CRC; 2021.
	
\bibitem{KhaiKir01}
Khamsi~MA, Kirk~WA. An introduction to metric spaces and fixed point theory. New York: John Wiley \& Sons; 2001.
	
\bibitem{Nad68}
Nadler~S. Sequences of contractions and fixed points. Pacific Journal of Mathematics. 1968;\hspace{0pt}27(3):579--585.
	
\bibitem{For82}
Fora~A. A fixed point theorem for product spaces. Pacific Journal of
Mathematics. 1982;\hspace{0pt}99(2):327--335.
	
\bibitem{KirSte84}
Kirk~WA, Sternfeld~Y. The fixed point property for nonexpansive mappings in certain product spaces. Houston Journal of Mathematics.
1984;\hspace{0pt}10(2):207--214.
	
\bibitem{TanXu91}
Tan~KK, Xu~HK. On fixed point theorems of nonexpansive mappings in product spaces. Proceedings of the American Mathematical Society.
1991;\hspace{0pt}113(4):983--989.
	
\bibitem{EspKir01}
Esp{\'i}nola~R, Kirk~WA. {Fixed points and approximate fixed points in product spaces}. Taiwanese Journal of Mathematics. 2001;\hspace{0pt}5(2):405 -- 416.
	
\bibitem{Kuc90}
Kuczumow~T. Fixed point theorems in product spaces. Proceedings of the American Mathematical Society. 1990;\hspace{0pt}108(3):727--729.
	
\bibitem{KohLeu07}
Kohlenbach~U, Leu{\c{s}}tean~L. The approximate fixed point property in product spaces. Nonlinear Analysis. 2007;\hspace{0pt}66(4):806--818.
	
\bibitem{BraMorScaTij03}
Br{\^a}nzei~R, Morgan~J, Scalzo~V, et~al. Approximate fixed point theorems in Banach spaces with applications in game theory. Journal of Mathematical
Analysis and Applications. 2003;\hspace{0pt}285(2):619--628.
	
\bibitem{LooSte90}
Loomis~LH, Sternberg~S. Advanced calculus. Jones and Bartlett Publishers; 1990.
	
\bibitem{ChaPet91}
Johnson~CR, Nylen~P. Monotonicity properties of norms. Linear Algebra and its Applications. 1991;\hspace{0pt}148:43--58.
	
\bibitem{BauStoWit61}
Bauer~FL, Stoer~J, Witzgall~C. Absolute and monotonic norms. Numerische
Mathematik. 1961;\hspace{0pt}3:257--264.
	
\bibitem{Pap14}
Papadopoulos~A. Metric spaces, convexity and nonpositive curvature. 2nd ed. Strasbourg, France: EMS Press; 2014. IRMA Monographs.
	
\bibitem{BriHae99}
Bridson~R~Martin\, Haefliger~A. Metric spaces of non-positive curvature. Berlin \& Heidelberg: Springer-Verlag; 1999.
	
\bibitem{Kir04.2}
Kirk~WA. Fixed point theorems in $ \text{CAT}$ spaces and $ \mathbb{R}$-trees. Fixed Point Theory and Applications. 2004;\hspace{0pt}4:309--316.
	
\bibitem{AriLiLop14}
Ariza-Ruiz~D, Li~C, López-Acedo~G. The Schauder fixed point theorem in geodesic spaces. Journal of Mathematical Analysis and Applications. 2014;\hspace{0pt}417(1):345--360.
	
\bibitem{KirSha17}
Kirk~W, Shahzad~N. Fixed points of locally nonexpansive mappings in geodesic spaces. Journal of Mathematical Analysis and Applications.
2017;\hspace{0pt}447(2):705--715.
	
\bibitem{KirSha14}
Kirk~WA, Shahzad~N. Fixed point theory in distance spaces. Cham: Springer International Publishing; 2014.
	
\bibitem{AriLeuLop14}
Ariza-Ruiz~D, Leu\c{s}tean~L, López-Acedo~G. Firmly nonexpansive mappings in classes of geodesic spaces. Transactions of the American Mathematical
Society. 2014;\hspace{0pt}366(8):4299--4322.
	
\bibitem{ReiZas24}
Reich~S, Zaslavski~AJ. {A fixed point theorem for nonself nonlinear
contractions in length spaces}. Topological Methods in Nonlinear Analysis. 2024;\hspace{0pt}63(1):13 -- 22.
	
\bibitem{Iof17}
Ioffe~AD. Variational analysis of regular mappings. {T}heory and applications.
Springer; 2017. Springer Monographs in Mathematics.
	
\bibitem{AzeCorLuc02}
Az{\'e}~D, Corvellec~JN, Lucchetti~RE. Variational pairs and applications to stability in nonsmooth analysis. Nonlinear Anal, Ser A: Theory Methods.
2002;\hspace{0pt}49(5):643--670.
	
\bibitem{CuoKru22}
Cuong~ND, Kruger~A. Error bounds revisited. Optimization.
2022;\hspace{0pt}71(4):1021--1053.
	
\bibitem{CuoKru21.3}
Cuong~ND, Kruger~AY. Uniform regularity of set-valued mappings and stability of implicit multifunctions. Journal of Nonsmooth Analysis and Optimization. 2021
Jun;\hspace{0pt}Volume 2:3.
	
\bibitem{Iof01+}
Ioffe~A. Towards metric theory of metric regularity. In: Lassonde~M, editor.
Approximation, Optimization and Mathematical Economics; Heidelberg.
Physica-Verlag HD; 2001. p. 165--176.
	
\bibitem{AzeCor04}
Az{\'e}~D, Corvellec~JN. Characterizations of error bounds for lower
semicontinuous functions on metric spaces. ESAIM: Control Optim Calc Var.
2004;\hspace{0pt}10(3):409--425.
	
\bibitem{SaiKatTak00}
Saito~KS, Kato~M, Takahashi~Y. Absolute norms on $\mathbb{C}^n$. Journal of
Mathematical Analysis and Applications. 2000;\hspace{0pt}252(2):879--905.
	
\bibitem{Cuo25}
Cuong~ND. Primal and dual characterizations of sign-symmetric norms ; 2025.
https://arxiv.org/abs/2504.19642.
	
\bibitem{Cuo26}
Cuong~ND. Dual characterizations of norm minimization problems ; 2026.
https://arxiv.org/abs/2601.08153.
	
\bibitem{AvgFon00}
Avgustinovich~S, Fon-Der-Flaass~D. Cartesian products of graphs and metric
spaces. European Journal of Combinatorics. 2000;\hspace{0pt}21(7):847--851.
	
\bibitem{Bai88}
Baillon~J. Nonexpansive mapping and hyperconvex spaces. In: Fixed point theory and its applications. (Contemporary Mathematics; Vol.~72). American
Mathematical Society; 1988. p. 11--19.
	
\bibitem{Bis23}
Bisht~RK. Fixed point of Lipschitz type mappings. Applied General Topology. 2023 Oct;\hspace{0pt}24(2):449–454.
	
\bibitem{Gor96}
G{\'o}rnicki~J. Fixed points of involutions. Math Japonica.
1996;\hspace{0pt}43(1):151--155.
	
\bibitem{Gor01}
G{\'o}rnicki~J. A survey of some fixed point results for Lipschitzian mappings in Hilbert spaces. Nonlinear Analysis: Theory, Methods \& Applications.
2001;\hspace{0pt}47(4):2743--2751.
	
\bibitem{YanKir88}
Yanez~CM, Kirk~WA. Nonexpansive and locally nonexpansive mappings in product
spaces. Nonlinear Analysis. 1988;\hspace{0pt}12(7):719--725.
	
\bibitem{BerRus20}
Berinde~V, Rus~IA. Asymptotic regularity, fixed points and successive
approximations. Filomat. 2020;\hspace{0pt}34(3):965--981.
	
\bibitem{Shi76}
Ishikawa~S. Fixed points and iteration of a nonexpansive mapping in a Banach space. Proceedings of the American Mathematical Society.
1976;\hspace{0pt}59(1):65--71.
	
\end{thebibliography}
%\bibliographystyle{spmpsci}
%\bibliographystyle{amsplain}
%\bibliographystyle{tfnlm}

\end{document}